\documentclass{amsart}
\usepackage{amsmath,amsfonts,amssymb,amscd,amsthm,epsf}
\usepackage{pinlabel}
\newtheorem{prop}{Proposition}[section]
\newtheorem{thm}[prop]{Theorem}
\newtheorem{cor}[prop]{Corollary}
\newtheorem{ques}[prop]{Question}
\newtheorem{lem}[prop]{Lemma}

\newtheorem{adden}[prop]{Addendum}

\theoremstyle{definition}
\newtheorem{de}[prop]{Definition}
\newtheorem{example}[prop]{Example}
\newtheorem{examples}[prop]{Examples}
\theoremstyle{remark}
\newtheorem{Remark}[prop]{Remark}             
\newtheorem{Remarks}[prop]{Remarks}             
\newtheorem*{remark}{Remark}            
\def\C{{\mathbb C}}
\def\CP{{\mathbb C \mathbb P}}
\def\Z{{\mathbb Z}}
\def\R{{\mathbb R}}
\def\Q{{\mathbb Q}}

\def\E{{\mathcal E}}
\def\emb{{\hookrightarrow}}
\def\int{\mathop{\rm int}\nolimits}
\def\cl{\mathop{\rm cl}\nolimits}

\def\dim{\mathop{\rm dim}\nolimits}

\def\im{\mathop{\rm Im}\nolimits}
\def\max{\mathop{\rm max}\nolimits}
\def\id{\mathop{\rm id}\nolimits}

\def\Rk{\mathop{\rm Rank}\nolimits}
\def\top{\mathop{\rm TOP}\nolimits}
\def\co{\colon\thinspace}
\def\Hinf{\mathop{\rm H_2^\leftarrow}\nolimits}
\def\limin{\mathop{\lim_\leftarrow}\nolimits}
\def\Self{\mathop{\rm Self}\nolimits}
\begin{document}
\title{Minimal genera of open 4--manifolds}
\author{Robert E Gompf}
\thanks{Partially supported by NSF grants DMS--0603958 and 1005304.}
\address{The University of Texas at Austin, Mathematics Department RLM 8.100, Attn: Robert Gompf,
2515 Speedway Stop C1200, Austin, Texas 78712--1202}
\email{gompf@math.utexas.edu}
\date{January 15, 2016}
\begin{abstract} We study exotic smoothings of open 4--manifolds using the minimal genus function and its analog for end homology. While traditional techniques in open 4--manifold smoothing theory give no control of minimal genera, we make progress by using the adjunction inequality for Stein surfaces. Smoothings can be constructed with much more control of these genus functions than the compact setting seems to allow. As an application, we expand the range of 4--manifolds known to have exotic smoothings (up to diffeomorphism). For example, every 2--handlebody interior (possibly infinite or nonorientable) has an exotic smoothing, and ``most" have infinitely, or sometimes uncountably many, distinguished by the genus function and admitting Stein structures when orientable. Manifolds with 3--homology are also accessible. We investigate topological submanifolds of smooth 4--manifolds. Every domain of holomorphy (Stein open subset) in $\C^2$ is topologically isotopic to uncountably many other diffeomorphism types of domains of holomorphy with the same genus functions, or with varying but controlled genus functions.
\end{abstract}
\maketitle


\section{Introduction}

Classification theory for smooth structures on a fixed topological manifold is entirely anomalous in dimension four. While the theory becomes trivial in dimensions three and below, and reduces to obstruction theory in dimensions five and higher, the 4--dimensional theory is much more subtle and complicated. The initial shock of discovery came in the early 1980s, when the foundational theories of Freedman \cite{F} and Donaldson \cite{D1} completed prior work of Casson \cite{C} to show that Euclidean 4--space $\R^4$ has exotic smoothings, even though obstruction theory is trivial on contractible manifolds. It soon developed that $\R^4$ and many other open 4--manifolds have uncountably many diffeomorphism types of smoothings, arising in continuous families unique to dimension four. This prompted a fundamental question that remains open three decades later:

\begin{ques}\label{Q} Does every open 4--manifold admit more than one diffeomorphism type of smoothing? Infinitely many? Uncountably many?
\end{ques}

\noindent While this question seems to have languished in recent years, its compact analog continues to be extensively studied. Compact 4--manifolds frequently have infinitely many smoothings (necessarily countably many), in defiance of the high dimensional obstruction theory. The compact theory, descended from Donaldson's pioneering work \cite{D2}, is very concrete, with explicitly specified smooth 4--manifolds distinguished by invariants that can often be computed and applied to basic topological problems. For example, the {\em adjunction inequality} for these invariants provides information about the minimal genus of smoothly embedded surfaces representing a given homology class. In contrast, the traditional theory for open 4--manifolds is very indirect. The smoothings are rarely specified explicitly, and are typically distinguished using proofs by contradiction. There is a shortage of computable invariants, or even explicit manifolds for which to compute them. The adjunction inequality is known for the class of open 4--manifolds called Stein surfaces, but outside of that context, minimal genera of embedded surfaces appear not to have been previously investigated on open 4--manifolds. The present article addresses these issues by systematically studying smoothing theory on open 4--manifolds via minimal genera. In contrast with the compact case, it seems much easier to construct smoothings on open 4--manifolds with prespecified control of minimal genera, and these genera can be packaged into powerful invariants. We make substantial progress on the above question, sometimes with explicitly computed invariants on explicit manifolds. The overall intent is to elucidate open 4--manifolds by importing compact techniques into the noncompact world.

 The background of the subject begins with the shadow of high dimensional smoothing theory falling on dimension four. In this theory, it is natural to classify smooth structures up to isotopy rather than just diffeomorphism: Two smooth structures on a fixed topological manifold $X$ with boundary are {\em isotopic} if there is a diffeomorphism between them that is topologically ambiently isotopic to the identity (ie the diffeomorphism is related to the identity by a homotopy through homeomorphisms). There is a canonical map from isotopy classes on $X$ to those on $X\times \R$, which is a bijection when $\dim X>4$. Thus, we can apply the high dimensional theory to 4--manifolds by passing to dimension five to consider {\em stable isotopy classes}, ie isotopy classes on $X\times\R$. Unlike in higher dimensions, a given stable isotopy class on a 4--manifold can be represented by many isotopy classes of smoothings, or by none at all. For open 4--manifolds, however, the map is surjective \cite{Q}, so every stable isotopy class is realized by a smoothing of the 4--manifold. Clearly, isotopic smoothings are always diffeomorphic. The converse fails, however, even in high dimensions. For smoothings on 4--manifolds, stable isotopy and diffeomorphism are independent relations. The techniques of this paper frequently distinguish infinitely many diffeomorphism types within a given stable isotopy class. The primary obstruction to the existence of a smooth structure, extending a given smoothing on $\partial X$, is the {\em Kirby--Siebenmann invariant} of \cite{KS} in ${\rm H}^4(X^n,\partial X^n;\Z_2)$. (We work rel boundary since the boundary of a 4--manifold necessarily has a unique isotopy class of smoothings.) When this obstruction vanishes, stable isotopy classes of smoothings on a 4--manifold are classified by ${\rm H}^3(X,\partial X;\Z_2)$. That is, the set of such classes is in bijective correspondence with the group, canonically once we specify which class corresponds to 0. See \cite{FQ} for further discussion on smoothing theory and its 4--dimensional idiosyncrasies, and \cite{GS} for background on many of the ideas used in this paper.

Section~9.4 of \cite{GS} summarized what was known about smooth structures on open 4--manifolds, and was current up to the work described in the present paper. Various results on continuous families of smooth structures (eg \cite{menag}, \cite{Fa}) were combined and generalized into Theorem~9.4.24 of that book. For an open, connected topological 4--manifold $X$, let $V\subset X$ be the complement of a compact subset, or a component with noncompact closure of such a complement. In the former case, assume $V$ is orientable. If $V$ has a smoothing for which some finite cover embeds smoothly in $\# n\C P^2$ for some finite $n$, then $X$ admits uncountably many isotopy classes of smoothings, in each stable isotopy class agreeing with a particular one on $V$. Under some circumstances, one can use this theorem to distinguish uncountably many diffeomorphism types within every stable isotopy class. For example, this occurs if some end of $X$ is collared by $M^3\times\R$ (ie $M$ is closed and $M\times[0,\infty)$ properly embeds in $X$), where $M$ has a bicollared topological embedding in $\# n\C P^2$, or if $M$ is orientable and has a finite cover smoothly embedding in $\# n\C P^2$ (Corollary~9.4.25). Many, but probably not all, closed 3--manifolds satisfy at least one of these conditions. In a different direction, Taylor's invariant \cite{T} (extending earlier work in \cite{BE} and previously \cite{menag}) stems from trying to embed $X$ into a closed, spin 4--manifold with signature 0, and minimizing the resulting $b_2$. (We suppress various crucial details here.) The resulting invariant, which for spin manifolds takes values in $\Z^{\ge 0}\cup\{\infty\}$, shows that if $X$ is orientable, admitting a smoothing with a proper Morse function $X\to[0,\infty)$ with only finitely many critical points of index 3, and if ${\rm H}_2(X)$ has finite dimension with both $\Z_2$ and $\Q$ coefficients, then each stable isotopy class contains (at least countably) infinitely many diffeomorphism types of smoothings, cf.\  \cite[Theorem~9.4.29(a)]{GS}. Note that none of the techniques discussed so far allow any control of surfaces representing specific homology classes.

The origin of the present paper is the remaining result in \cite{GS}, ie Theorem~9.4.29(b). For this, we assume that $X$ is orientable, ${\rm H}_2(X)$ is nonzero, and there is a proper Morse function $X\to[0,\infty)$ with indices at most 2. (Throughout the text, we use homology with integer coefficients except where otherwise specified.) These hypotheses allow us to access the adjunction inequality via Stein surfaces and Freedman theory. The theorem concludes that $X$ admits (at least countably) infinitely many isotopy classes of smooth structures (realizing the unique stable isotopy class). Over time, the author expanded this method in various directions (previously unpublished). Recently, Akbulut and Yasui \cite{AY}, \cite{Y} have applied related techniques to distinguish infinite families of smoothings on compact 4--manifolds with boundary, with arbitrarily large finite subfamilies admitting Stein domain structures. (Their examples are constructed with finitely many handles, whereas the examples in the present paper come from infinite constructions involving Casson handles, as is typical of noncompact smooth 4--manifold constructions.)

The present paper uses Freedman theory to construct exotic smooth structures (smoothings not diffeomorphic to a preassigned one), and the adjunction inequality for Stein surfaces to systematically study these via the smoothly embedded surfaces they contain. We obtain various new existence theorems, such as this corollary of Theorems~\ref{inf} and \ref{cantor}:

\begin{thm}\label{exotic} If a smooth, open 4--manifold admits a proper Morse function $X\to[0,\infty)$ with indices at most 2, then it admits an exotic smooth structure. It admits infinitely many diffeomorphism types of smoothings if ${\rm H}_2(X)\ne 0$ or $X$ is not a $K(\pi,1)$. It admits uncountably many if ${\rm H}_2(X)$ has infinite rank.
\end{thm}

\noindent For orientable $X$, the countably infinite case ${\rm H}_2(X)\ne 0$ is \cite[Theorem~9.4.29(b)]{GS} upgraded from isotopy classes to diffeomorphism types. The infinite collections of smooth manifolds produced here all admit Stein structures when $X$ is orientable. Dropping the condition on Morse indices, we have (cf.\ \cite{BE}):

\begin{thm}\label{introM}
Let $X$ be a connected topological 4--manifold (possibly with boundary) with some end collared by $M\times\R$ for a closed, connected 3--manifold $M$. If $X-M\times(0,\infty)$ is compact, assume its Kirby--Siebenmann invariant vanishes. Let $\widetilde M$ denote $M$ (if orientable) or its orientable double cover.  If  ${\rm H}_2(\widetilde M)\ne0$, then $X$ has infinitely many diffeomorphism types of smoothings.
\end{thm}

\noindent See Theorem~\ref{collar}. This is not new, since Bi\v zaca and Etnyre \cite{BE} proved that every 4--manifold with a collared end admits infinitely many diffeomorphism types of smoothings. (Their missing case of infinitely many homeomorphic collared ends was completed by an observation of Ganzell \cite{Ga}.) However, their appproach gives no control of genera of surfaces. Furthermore, Theorem~\ref{introM} is a special case of the much more general Theorem~\ref{endsystemdiff} that also covers examples such as the following (Theorem~\ref{1h}(b)):

\begin{thm}\label{intro1h}
If $X$ admits a Morse function as above with indices at most 1, then $X$ has infinitely many diffeomorphism types of smoothings. If ${\rm H}_1(X)$ is not finitely generated, there are uncountably many.
\end{thm}

\noindent There are other approaches to this latter theorem: Taylor's invariant distinguishes infinitely many diffeomorphism types, and for orientable $X$,  \cite[Theorem~9.4.24]{GS} gives uncountably many isotopy classes (or diffeomorphism types if ${\rm H}_1(X)$ {\em is} finitely generated). What is most interesting here is that when  ${\rm H}_1(X)\ne 0$, this theorem follows just by studying genera of embedded surfaces, even though ${\rm H}_2(X)=0$.

While these theorems include substantial progress on Question~\ref{Q}, they are mainly presented as applications of controlling minimal genera, which is the main theme of the paper. To this end, Section~\ref{Background} presents our invariants derived from the genus function of a 4--manifold, along with required background material on Casson handles, Stein surfaces and the adjunction inequality. Our Main Lemma~\ref{main} provides control of these invariants for smoothings of a handlebody interior with indices at most 2, and the rest of Section~\ref{Main} provides simple applications, such as Theorem~\ref{exotic}. For deeper applications, we analyze ends of manifolds with the {\em genus function at infinity}. This is sensitive enough to detect exotic smoothings on manifolds with trivial 2--homology such as in Theorem~\ref{intro1h}. Our main lemma for controlling the genus function at infinity is presented in Section~\ref{Ends} and applied in Section~\ref{Apps}. For simplicity, we focus on collared ends (eg Theorem~\ref{introM}), although the techniques apply in much more generality. Again, the main theme is flexible control of the genus function. We see, for example, that smoothings on $S^1\times\R^3$ realize all values in $\Z^{\ge0}\cup\{\infty\}$ of the minimal genus of the generator at infinity (Theorem~\ref{1h}(a)). The remaining sections are nearly independent of each other. Section~\ref{Submfd} illustrates how badly topological submanifolds and isotopies can fail to respect smooth structures in dimension 4. We define a notion of minimal genus function for arbitrary subsets of a smooth 4--manifold, and show that for tame topologically embedded surfaces and 3--manifolds, this function can be flexibly changed by topological isotopy. In the orientable case, for example, a smooth surface can be topologically adjusted to realize any minimal genus exceeding the original. In Section~\ref{Uncountable}, we compare our genus invariants with older methods of detecting exotic smooth structures. We see that under various hypotheses, our previous examples occur in uncountable families for each choice of the genus functions. The theorems of this section, like those of Section~\ref{Main}, are compatible with Stein structures. This allows us to prove a theorem related to classical complex analysis in $\C^2$. A {\em domain of holomorphy} is an open subset of $\C^2$ (or of another Stein manifold) for which the inherited complex structure is Stein. Corollary~\ref{holomorphy}(a) implies:

\begin{thm}\label{introHolomorphy} For every domain of holomorphy $U$ in $\C^2$ (for example), the inclusion map is topologically isotopic to other embeddings, whose images are also domains of holomorphy and represent uncountably many diffeomorphism types of smoothings of $U$, while having the same genus function.
\end{thm}

\noindent We can also flexibly control the genus function here. Section~\ref{Related} further explores the range of the genus function, and shows how additional information can sometimes be obtained by allowing our surfaces to be immersed.

This paper is a streamlined version of the preprint \cite{MinGen}, which was refereed without mathematical difficulty. For improved readability, some text has been simplified or eliminated, with occasional references indicating what the interested reader can find in the original preprint. Examples~\ref{mess} and \ref{mess2} on noncollared ends, and Theorem~\ref{CH} on Casson handles, are new in this version. The final example of \cite{MinGen} has been deleted. This was posed as an example of an open 4--manifold not known to admit exotic smoothings. However Julia Bennett has observed that exotic smoothings can be constructed on it and detected by the Taylor invariant, since it has ``few essential 3--handles" as in \cite{T}. She has also shown \cite{B} that the Taylor invariant can often be controlled independently of the constructions of this paper, by augmenting Section~\ref{Uncountable} with a novel exotic $\R^4$.


\section{Preliminaries}\label{Background}

We begin by reviewing the genus function and extracting the invariants that we will need. Then we provide background on handlebodies, Stein surfaces, the adjunction inequality and Casson handles.

For any smooth 4--manifold $X$ (possibly with boundary), every homology class $\alpha\in {\rm H}_2(X)$ is represented by a smoothly embedded, closed, oriented surface $F$. (We will frequently refer to such an $F$ as a ``surface representing $\alpha$'', suppressing the remaining adjectives.) The genus $g(F)$ of $F$ can be increased by any positive integer, simply by adding small tubes. One cannot always decrease the genus, however.

\begin{de}\label{genus function} The {\em genus function} $G\co {\rm H}_2(X)\to \Z^{\ge 0}$ is the function assigning to each $\alpha$ the smallest possible genus of a surface $F$ representing $\alpha$.
\end{de}

\noindent We define the genus of a disconnected surface to be the sum of the genera of its components. Then we can always arrange $F$ to be connected, provided that $X$ is. For closed, oriented 4--manifolds, and subsequently for Stein surfaces, the genus function has been studied using Donaldson invariants and their descendants. General open 4--manifolds, however, seem largely unexplored by this method. We will show that the genus function is especially powerful in this setting, due to an obvious inequality: For any smooth embedding $i\co X\to Y$ of 4--manifolds, we have $G(i_*\alpha)\le G(\alpha)$. (We do not require embeddings to be proper.) This allows us to constrain the genus function on $X$ by exploiting Donaldson-type invariants on many manifolds $Y$ simultaneously.

The genus function is an invariant of smooth structures on a 4--manifold $X$ up to isotopy, or more generally up to diffeomorphisms fixing ${\rm H}_2(X)$, but it is not preserved by more general diffeomorphisms between smoothings of $X$. Nevertheless, many diffeomorphism invariants can be constructed from it. One approach we will use is the following: Suppose $X$ is a smooth, oriented 4--manifold with torsion subgroup $T\subset {\rm H}_2(X)$. For $g\in\Z^{\ge 0}$, let $\Gamma_g=\Gamma_g(X)\subset {\rm H}_2(X)/T$ denote the rational span of $G^{-1}[0,g]\cap Q^{-1}[-g,g]$, where $Q$ is the intersection form on ${\rm H}_2(X)$. That is, $\Gamma_g$ is the set of elements of ${\rm H}_2(X)/T$ realized up to scale as linear combinations of surfaces with genus and self-intersection bounded (in absolute value) by $g$. These groups $\Gamma_g$ are nested as $g$ increases, so we obtain a nondecreasing function $\gamma(g)=\Rk(\Gamma_g)$, at least as long as ${\rm H}_2(X)/T$ is free abelian. (This is always true if there is a handlebody whose interior is homeomorphic to $X$, and has only finitely many 3--handles.)

\begin{de}\label{genus filtration} We will call the function $\gamma\co\Z^{\ge 0}\to\Z^{\ge 0}\cup\{\infty\}$ the {\em genus--rank function} of $X$ (or of its given smoothing). The integers $g$ such that $\gamma(g)\ne \gamma(g-1)$ (where $\gamma(-1)=0$) will be called the {\em characteristic genera} of $X$, and the corresponding subgroups $\Gamma_g$ will be called the {\em genus filtration} of $X$.
\end{de}

\noindent Clearly, any diffeomorphism $X\to Y$ (between oriented 4--manifolds but possibly reversing orientation) sends the genus filtration of $X$ to that of $Y$, and the genus--rank function is a diffeomorphism invariant of smooth structures on $X$, determined by the genus function (and intersection form up to sign).

Every smooth manifold $X$ admits a proper Morse function $f\co X\to[0,\infty)$, although infinitely many critical points may be required. In the finite case, it is well known that such a function exhibits $X$ as the interior of a finite handlebody $H$. Similarly, in the infinite case, $X$ is the interior of an infinite handlebody $H$ with the maximal index of its handles equal to that of the critical points of $f$. (See \cite{yfest} for a proof.) We require handles to be attached in order of increasing index, and infinitely many 0--handles are required so as to avoid clustering of attaching regions on compact boundaries. We can now state our theorems in the language of handle theory without loss of generality. We call $H$ a {\em $k$--handlebody} if its handles have index at most $k$, with the case $k=2$ of special interest. Then ${\rm H}^3(X;\Z_2)=0$, so there is a unique stable isotopy class of smoothings. Furthermore, there are no 3--chains, so ${\rm H}_2(X)$ is free abelian (although not necessarily finitely generated). If $H_0$ is a subhandlebody of a 2--handlebody $H$, ie $H$ is made from the handlebody $H_0$ by attaching handles (to handles of lower index), the long exact sequence of the pair shows that ${\rm H}_2(H_0)$ is a direct summand of ${\rm H}_2(H)$. The handle structure is part of the defining data of a handlebody. We always assume (without loss of generality in dimension 4) that the attaching maps defining a handlebody are smooth, although we sometimes topologically embed handlebodies or otherwise construct exotic smoothings on their interiors.

Stein surfaces are complex surfaces (hence oriented 4--manifolds) arising as closed subsets of $\C^N$. They have a long and continuing history in complex analysis, but we will only need a few basic facts about them. (See \cite{GS} for more details and \cite{CE} for a broader perspective.) By work of Eliashberg (implicit in \cite{E}, see also \cite{CE}), a smooth, oriented, open 4--manifold $X$ admits a Stein structure if and only if it has a proper Morse function $X\to[0,\infty)$ with indices at most $2$ and a certain framing condition satisfied. This, in turn, is equivalent (by \cite[Theorem~A.2]{yfest}  in the infinite case) to exhibiting $X$ as the interior of a 2--handlebody with suitably framed 2--handles. The practical effect is that for a 1--handlebody $H_0$, oriented as a 4--manifold, and a proper embedding into $\partial H_0$ of a disjoint union of circles, the latter can always be framed so that the handlebody $H$ obtained by adding 2--handles to these framed circles has Stein interior (inducing the preassigned orientation). Given any such framing, any other framing obtained from it by adding left twists will also determine a manifold admitting a Stein structure. However, adding enough right twists typically produces a manifold that admits no Stein structure. The most basic invariant of a Stein surface $S$ is the Chern class $c_1(S)\in {\rm H}^2(S)$ of its tangent bundle with the induced complex structure. If $S$ is exhibited as the interior of a handlebody $H$ as above, then this class is represented by a {\em Chern cocycle} for the cellular cohomology of the associated CW 2--complex  (with a $k$--cell for each $k$--handle of $H$). The Chern cocycle is well-defined once we fix a complex trivialization of the tangent bundle over the 0-- and 1--handles, for example, by drawing a suitable diagram for the handlebody. Its value on each oriented 2--cell is called the {\em rotation number} of the corresponding (suitably oriented) attaching circle, and can be read from a diagram. If we change $H$ by removing one 2--handle and replacing it, with an additional left twist in the framing, the corresponding rotation number will change by $\pm 1$. We can choose the sign arbitrarily, so we obtain two Stein structures on the same handlebody interior, with two different Chern cocycles. If these cocycles determine different cohomology classes, then the Stein structures are distinguished by their complex tangent bundles (ie they determine almost-complex structures that are not homotopic).

We can now state the adjunction inequality for Stein surfaces. The following version is adapted from Nemirovski \cite{N}.

\begin{thm}\label{adj} Let $F$ be a closed, oriented surface of genus $g(F)$, generically smoothly immersed in a Stein surface $S$, with $k$ positive double points (and some number of negative double points). If $F$ is not a nullhomotopic sphere (or disjoint union thereof), then
$$2g(F)+2k-2\ge F\cdot F+|\langle c_1(S),F\rangle |.$$
\end{thm}

\noindent In \cite{N} and elsewhere in the literature, $F$ is assumed to be connected, but the statement easily generalizes to the disconnected case \cite{MinGen}. The latter is useful, for example, for dealing with multiple ends of a manifold. The version in \cite{N} also tracks some negative double points. We do not need that version here, and outside of our final Section~\ref{Related}, only need the embedded version (so $k=0$). The inequalities under discussion descend from a long line of adjunction inequalities, originating on closed manifolds via gauge theory pioneered by Donaldson and subsequently upgraded to Seiberg--Witten theory. For further references and a current exposition, see \cite{Fo}.

To construct exotic smooth structures, we will need to replace 2--handles by Casson handles \cite{C}, \cite{F}. Let $h$ be a 2--handle attached to a 4--manifold $X$ along a circle $C\subset\partial X$, and let $D\subset h$ be a generically immersed 2--disk with $\partial D=C$. Then the singularities of $D$ are all transverse double points, which have well-defined signs once an orientation for the 4--manifold $h$ is specified. Let $T_1$ be a compact regular neighborhood of $D$, which we think of as an oriented 4--manifold abstractly attached to $X$ in place of $h$, along the same framed circle. When $D$ is not embedded, we will call $T_1$ a {\em kinky handle} or {\em 1--stage tower} with {\em core} $D$ and {\em attaching region} $T_1\cap X$. The oriented diffeomorphism type of the pair $(T_1,C)$ is determined by the numbers of double points of each sign. In fact, $T_1$ is obtained from a 2--handle $h^*$ by self-plumbing to create the given numbers of double points of each sign in the core disk, although the framing for attaching $h^*$ is obtained from that of $h$ by subtracting twice the signed count $\Self(D)$ of the double points. (The canonical framing of $C$ in $h$ is the unique framing determining a parallel push-off $C'$ of $C$ for which $C$ and $C'$ bound disjoint surfaces in $h$. The framing retains this property in $T_1$. However, the canonical framing of $h^*$ is determined by a parallel push-off of $D$ in $T_1$, which intersects $D$ in $2\Self(D)$ points, counted with sign, and to compensate we must subtract this number from the original framing to get the framing for attaching $h^*$.) There is a canonical local procedure for adding a double point of either sign to an immersed surface in a 4--manifold, and repeatedly applying this to the core of $h$ gives a disk $D$ as above realizing any given numbers of double points of each sign. Furthermore, the local procedure is reversible in that we can find a 2--handle attached to $T_1$ in $h$ near each double point, so that attaching the 2--handles changes $T_1$ back into $h$. If we replace these 2--handles by kinky handles inside them, we obtain a {\em 2--stage tower} $T_2$. Iterating the construction, we obtain {\em $n$--stage towers} for all $n\in\Z_+$, $T_1\subset T_2\subset T_3\cdots$, each obtained from the previous one by attaching kinky handles to a suitable framed link. The union of these towers, with all boundary outside of $X$ removed, is called a {\em Casson handle $CH$} attached to $X$ with {\em attaching circle} $C$ (framed as for $h$ and $T_1$). By construction, $CH$ has a standard embedding in the 2--handle $h$. Note that while each tower has free fundamental group generated by the top-stage double points, the inclusion maps are $\pi_1$--trivial, so $CH$ is simply connected. In fact, Freedman proved that every Casson handle is homeomorphic (rel the framed attaching circle) to an open 2--handle $D^2\times\R^2$, which led immediately to revolutionary developments such as his classification of closed, simply connected, topological 4--manifolds (\cite{F}, as sharpened by Quinn \cite{Q}). The failure of such results in the smooth category, as first discovered by Donaldson, implies that Casson handles are typically not diffeomorphic to open 2--handles (cf.\ Theorem~\ref{CH}). Thus, we obtain exotic smoothings of  $D^2\times\R^2$ by pulling back the smoothings of Casson handles via Freedman homeomorphisms. Replacing 2--handles by Casson handles in a 2--handlebody can now be thought of as creating new smooth structures on its interior.

Casson handles are classified, up to orientation-preserving diffeomorphism respecting the subtowers, by based, signed trees with no finite branches. To define the tree for a given Casson handle, start with a vertex for each kinky handle, with the base point corresponding to $T_1$. Then add an edge for each pair of kinky handles that are directly attached to each other, labeled with the sign of the associated double point. We obtain a bijection sending Casson handles to based signed trees where only the base point is allowed to have valence one. If the based, signed tree for $CH$ is contained in that of $CH'$, then $CH'$ is called a {\em refinement} of $CH$, and by construction, there is an orientation-preserving embedding $CH'\subset CH$ respecting the framed attaching circles (and attaching regions) of the Casson handles. Any finite collection of Casson handles has a common refinement, the quickest construction being to identify the base points of their trees.

Our main tool for constructing exotic smooth structures for which the adjunction inequality can be applied is the following theorem:

\begin{thm}{\rm \cite{Ann}}\label{ann} Every oriented 2--handlebody has interior $X$ orientation-preserving homeomorphic to a Stein surface, whose Chern class pulls back to a preassigned lift of the Stiefel--Whitney class $w_2(X)$.
\end{thm}

\noindent The main idea is that replacing a 2--handle by a kinky handle with $\Self(D)>0$ is equivalent to replacing it by a new 2--handle with a more left-twisted framing, and then doing self-plumbings. By choosing each $\Self(D)$ sufficiently large, we can arrange the result of attaching the new 2--handles to have Stein interior, and the self-plumbings do not disturb the Stein condition. Iterating the construction, we can build a Stein surface with Casson handles in place of the original 2--handles, and this is homeomorphic to $X$ by Freedman's theorem. Choosing each first-stage $\Self(D)$ even larger provides enough flexibility to realize any preassigned Chern cocycle, yielding the required Chern class, and more generally, any preassigned homotopy class of almost-complex structures \cite{Ann}. The necessary values of  $\Self(D)$ at the first stage depend on the original handlebody and Chern cocycle, and can be computed from a diagram. At higher stages it suffices to use any kinky handles with $\Self(D)>0$, so a single double point (positive) is sufficient.

\section{Controlling the genus function}\label{Main}

The simplest application of Theorem~\ref{ann} to smoothing theory is to construct Stein exotic smoothings on a 2--handlebody interior and distinguish them by the adjunction inequality. We gain extra generality, allowing nonorientable manifolds for example, by passing to a cover first. For this reason, our convention throughout the text is {\em not} to assume orientations on manifolds, unless otherwise stated such as in reference to Stein surfaces or negative definite manifolds, and not to assume that maps such as homeomorphisms preserve orientation. The smoothings resulting from the main Lemma~\ref{main} of this section are (in principle) completely explicit, with the complexity of the Casson handles determined as in Theorem~\ref{ann}. In preparation, we introduce some useful terminology.

\begin{de}
Let $H$ be a 2--handlebody. We will say a smooth structure on $\int H$ has {\em Casson type} if it is obtained from the standard smooth structure by replacing some of the 2--handles with Casson handles. If $H$ is oriented, and the smoothing admits a compatibly oriented Stein structure constructed as in Theorem~\ref{ann}, we will say $\int H$ has {\em Stein--Casson type}.
\end{de}

\noindent Both of these properties pull back under covering maps, as does being a refinement of some Stein--Casson smoothing (ie being a Casson-type smoothing obtained from a Stein--Casson smoothing by refining each Casson handle). Note that by definition, the above smooth and Stein structures are standard on the 0-- and 1--handles, which avoids technicalities with removing handles or attaching new ones.

Our main lemma allows us to flexibly control the genus filtration and characteristic genera for Casson-type smoothings, as well as for their covers. Recall that in a handlebody $H$, a {\em subhandlebody} is a subset consisting of a union of handles of $H$ that themselves comprise a handlebody. There will typically be many ways to write $H$ as a nested union of subhandlebodies, and further flexibility results if we first modify $H$ by handle moves. For a fixed $n\in\{2,3,\dots,\infty\}$, consider a nested sequence $H_1\subset H_2\subset\cdots\subset H_n$ of subhandlebodies of a 4--dimensional 2--handlebody $H_n$, each with possibly infinitely many handles. If $n=\infty$, assume $H_\infty=\bigcup_{i=1}^\infty H_i$. Let $\pi\co \widetilde H_n\to H_n$ be a covering. For $i=1,\dots,n$, let $X_i=\int H_i$ and $\widetilde X_i=\int \widetilde H_i=\pi^{-1}(X_i)$, and let $A_i$ be the image of ${\rm H}_2(\widetilde X_i)$ in $A_n={\rm H}_2(\widetilde X_n)$. (Recall that $A_i\cong{\rm H}_2(\widetilde X_i)$ is a direct summand of the free abelian group $A_n$.)

\begin{lem}\label{main}
Suppose that $\widetilde X_n$ is orientable and that each $A_i$ with $i<n$ has finite rank. Then there is a smooth structure $\Sigma$ on $X_n$ for which each $A_i$ lies in the genus filtration of $\widetilde X_n$. In fact, there is an increasing sequence $\{k_i\thinspace |\thinspace 1\le i<n\}$ of integers such that, for $1\le i<n$, each $A_i=\Gamma_{k_i}(\widetilde X_n)$ equals the span in  $A_n$ of all surfaces $F$ in $\widetilde X_n$ that are smoothly embedded with respect to $\pi^*\Sigma$ and have $g(F)$ and $|F\cdot F|\le k_i$. The sequence $\{k_i\}$ can be chosen to increase arbitrarily rapidly from an arbitrarily large $k_1$.
\end{lem}

\noindent This holds more generally \cite{MinGen} with each $A_i$ replaced by $A_i/A_0$, where $A_0={\rm H}_2(\widetilde X_0)$ for some subhandlebody $H_0\subset H_1$. The extra generality is useful when each $A_i/A_0$ is finitely generated but $A_0$ is an infinite rank homeomorphism-invariant subgroup (eg Example~\ref{bundles}(b)). Note that when $X_1$ is empty, $k_1$ provides a strict lower bound for minimal genera of all nontrivial classes $\alpha$ in ${\rm H}_2(\widetilde X_n)$ with $|\alpha\cdot\alpha|\le k_1$.

To prove the lemma by induction, and for the main lemma of Section~\ref{Ends}, we must be able to preserve a preassigned smoothing on one of the given subhandlebody interiors $X_m$. (This also has immediate consequences such as \cite[Theorem~3.9]{MinGen}.) For simplicity, we assume that each $X_i$ is connected and that $X_1$ is not orientable unless $X_n$ is. (These hypotheses can always be arranged when $X_n$ is connected, by including more 1--handles in each $X_i$. The subgroups $A_i$ then remain unchanged.) Let $\widehat X_i$ denote $X_i$ if it is orientable, and its orientable double cover otherwise. Then $\pi$ factors through $\widehat X_n$. Fix an orientation on $\widehat X_n$ and lift it to $\widetilde X_n$.

\begin{adden}\label{mainadd}
The smoothing $\Sigma$ given by the lemma has Casson type, lifting to a refinement of some Stein--Casson smoothing $\widehat \Sigma$ on $\widehat X_n$ that equals $\Sigma$ when $\widehat X_n=X_n$ and realizes a preassigned homotopy class of almost-complex structures on $\widehat X_n$. The smoothings $\Sigma$ and $\widehat \Sigma$ can be chosen arbitrarily over $X_1$, subject to the conditions of the previous sentence. Similarly, for fixed $m$ with $1<m<n$, they can be chosen to agree with corresponding smoothings previously constructed for $X_m$ by the lemma and addendum, with the same values of $k_1,\dots,k_{m-1}$. Furthermore, $\Sigma$ can be chosen to also be a refinement of a preassigned Casson-type smoothing of $X_n$ (that extends the preassigned one on $X_m$ if given, $m\ge1$).
\end{adden}

\begin{proof}[Proof of Lemma~\ref{main} and Addendum~\ref{mainadd}] If $\Sigma$ and  $\widehat\Sigma$ were not preassigned over $X_1$, we must first construct them. If $\widehat X_1=X_1$, let $\widehat\Sigma=\Sigma$ be any Stein--Casson smoothing of $X_1$ respecting the given orientation and almost-complex structure, and refining the preassigned smoothing. Otherwise, choose a Stein--Casson smoothing $\widehat\Sigma$ on $\widehat X_1$ suitably respecting the data. Each 2--handle $h$ of $X_1$ has two lifts to $\widehat X_1$, inheriting two different Casson handle structures from  $\widehat\Sigma$. Because of the orientation mismatch, orienting $h$ results in one Casson handle having excess positive double points, and the other having excess negativity. Define $\Sigma$ so that $h$ is a common refinement of these two Casson handles (and of the preassigned one if given). 

To complete the proof when $n=2$, extend  $\widehat \Sigma$ to a Stein--Casson smoothing of $\widehat X_2$. Choose a finite collection $\{F_r\}$ of oriented surfaces in $\widetilde X_1$, smoothly embedded with respect to $\pi^*\Sigma$, and spanning $A_1$. Choose any integer $k_1$ larger than each $g(F_r)$ and $|F_r\cdot F_r| $. Let $\{h_l\}$ be the set of 2--handles of $\widetilde H_2$ disjoint from $\widetilde X_1$, oriented as 2--chains lifting a fixed choice of orientations on  $H_2$. For each $l$ and choice of sign, obtain a Stein surface $S_l^\pm$ homeomorphic to $\widetilde X_2$, by lifting  $\widehat \Sigma$ from $\widehat X_2$ to $\widetilde X_2$ and then refining $h_l$, leaving its complement unchanged, so that the Chern cocycle $r(S_l^\pm)$ on $h_l$ is bounded away from 0 by $\pm3k_1$. For fixed sign, we can use the same refinement for each $h_l$ over a given 2--handle of $\widehat X_2$. This allows us to extend $\Sigma$ to a Casson-type smoothing of $X_2$ with each 2--handle interior smoothed as a refinement of each of its lifts to the corresponding Stein surfaces $S_l^\pm$ (and a refinement of the preassigned smoothing if given). Clearly, $A_1$ is spanned as required by the smooth surfaces $\{F_r\}$. However, any cycle representing a homology class outside of $A_1$ has the form $mh_l+\alpha$ for some $l$, where $m$ is nonzero and $\alpha$ is the contribution from other 2--handles. If $F$ is a $\pi^*\Sigma$--smooth surface representing this class with $|F\cdot F|\le k_1$, then it includes smoothly in the two corresponding Stein surfaces $S_l^\pm$. By the adjunction inequality, $2g(F)-2\ge F\cdot F+|\langle c_1(S_l^\pm),F \rangle|\ge-k_1+|m\langle r(S_l^\pm),h_l \rangle+\langle r(S_l^\pm),\alpha \rangle|\ge 2k_1$ if we choose the sign for $S_l^\pm$ so that the two terms in the absolute value bars have the same sign. (Note that the latter of these is independent of the sign.) In particular, $g(F)>k_1$ as required. If $\widehat X_2=X_2$, redefine $\widehat\Sigma$ to be $\Sigma$, which we can assume has Stein--Casson type.

 For $n>2$, we apply induction. For a given $i>2$, we assume $X_{i-1}$ has already been smoothed by the lemma and addendum. We apply the $n=2$ version to the pair $X_{i-1}\subset X_i$, suitably extending the smoothing from $X_{i-1}$ to $X_i$ using the first sentence of the addendum. This extends our previous $\widehat\Sigma$, and hence our previous Stein surfaces $S_l^{\pm}$. Together with the newly constructed Stein surfaces, these give the required lower bounds on minimal genera. (Without reusing the old Stein surfaces $S_l^{\pm}$, it would be conceivable that the new handles of $X_i$ could lower the minimal genus of a class in $A_2-A_1$ below $k_1$, for example.) If $n=\infty$, the induction gives a smoothing on each $X_i$ with $i$ finite. These all agree on their overlaps, so we obtain a smoothing on $X_\infty$, which has the required properties by compactness of the relevant surfaces.
\end{proof}

We can now create many exotic smooth structures, using control of the genus function to distinguish them. The rest of this section illustrates the method with some sample applications. The subsequent two sections apply the lemma more deeply to study minimal genera at infinity.

\begin{thm}\label{inf} Let $X$ be the interior of a (connected) 2--handlebody. Then $X$ admits more than one diffeomorphism type of smooth structure. It admits infinitely many provided that one of the following conditions holds:
\item[a)] ${\rm H}_2(X)\ne 0$,
\item[b)] $X$ is nonorientable and its orientable double cover $\widetilde X$ has ${\rm H}_2(\widetilde X)\ne 0$,
\item[c)] $X$ is not a $K(\pi,1)$,
\item[d)] $X$ has an orientable (connected) cover $\widetilde X$ with ${\rm H}_2(\widetilde X)\ne 0$, and such that $\pi_1(\widetilde X)\subset\pi_1(X)$ has only finitely many images under the homeomorphism group of $X$, up to inner automorphism.
\end{thm}

\noindent The orientable case of (a) is proved up to isotopy in \cite[Theorem~9.4.29(b)]{GS} by a simpler application of the same idea. The $\pi_1$ condition in (d) is always true for a cover $\widetilde X$ of finite degree $d$ when $\pi_1(X)$ is finitely generated, for then the latter is realized by a 2--complex with finite 1--skeleton, which has only finitely many $d$--fold covers. The smoothings of (a--d) have Stein--Casson type when $X$ is oriented.

\begin{proof} For nonorientable $X$, Hypothesis (a) implies (b), since $b_2(\widetilde X)\ge b_2(X)$. The covering map in (b) is uniquely determined by the homeomorphism type of $X$, as is the universal covering in (c), so (a), (b) and (c) are all special cases of (d). In that case, the hypothesis on $\pi_1$ guarantees that $X$ has finitely many coverings $\widetilde X_m\to X$ such that every self-homeomorphism of $X$ lifts to a homeomorphism $\widetilde X\to\widetilde X_m$ for some $m$. Let $\Sigma$ be a smoothing of $X$. For each $m$, ${\rm H}_2(\widetilde X_m)\cong {\rm H}_2(\widetilde X)\ne 0$, so we can choose a homologically essential surface in $\widetilde X_m$ that is smooth with respect to the lift of $\Sigma$. Let $g_m$ and $q_m$ be its genus and self-intersection. Lemma~\ref{main} with $n=2$ and $H_1$ empty gives us another smoothing $\Sigma'$ of $X$ for which every smooth essential surface $F$ in its lift to $\widetilde X$ has $g(F)$ or $|F\cdot F|>\max_m\{g_m,|q_m|\}$. There can be no diffeomorphism from $X_{\Sigma'}$ to $X_\Sigma$, for this would be a self-homeomorphism of $X$. Lifting to  a homeomorphism $\widetilde X\to\widetilde X_m$ for some $m$, we would obtain a diffeomorphism between the corresponding lifts of $\Sigma'$ and $\Sigma$. But by construction, the latter lift has an essential surface of genus $g_m$ and self-intersection $q_m$, whereas the former cannot, yielding the required contradiction. Now repeat the entire construction with $\Sigma'$ in place of $\Sigma$, and inductively obtain a sequence of nondiffeomorphic smooth structures on $X$.

To see that $X$ always admits more than one smooth structure, it now suffices to consider the case when $X$ is a $K(\pi,1)$. In that case, the universal cover $\widetilde X$ is contractible and has no 3--handles, so by  \cite[Theorem~4.3]{T}, the standard smooth structure on $\widetilde X$ has vanishing Taylor invariant. If we form the end-sum of $X$ with an exotic $\R^4$ (cf.\ Section~\ref{Uncountable}) whose Taylor invariant is nonzero, then the universal cover will have nonzero (possibly infinite) Taylor invariant. Thus, we have two smooth structures on $X$ with nondiffeomorphic universal covers.
\end{proof}

Note that unlike our previous orientable examples, the exotic smooth structure we obtain when (a--d) fail cannot be Stein or Casson type. This is because its universal cover has nonzero Taylor invariant, so admits no handle decomposition without 3--handles. In fact, any handle decomposition of the exotic $\R^4$ summand requires infinitely many 3--handles \cite{T}, and no such decomposition is explicitly known, making this case the only nonconstructive proof in this section.

It should not be surprising that when ${\rm H}_2( X)$ has infinite rank, the genus--rank function distinguishes uncountably many diffeomorphism types of smoothings, and these are Stein--Casson if $X$ is oriented. The proof merely requires careful bookkeeping. More generally, we have the following \cite{MinGen}, where the $\pi_1$--condition again is automatic if $\pi_1(X)$ is finitely generated:

\begin{thm}\label{cantor} Let $X$ be the interior of a 2--handlebody. Suppose $X$ has an orientable finite cover $\widetilde X$ with ${\rm H}_2(\widetilde X)$ not finitely generated, and such that $\pi_1(\widetilde X)\subset\pi_1(X)$ has only countably many images under the homeomorphism group of $X$, up to inner automorphism. Then $X$ admits uncountably many diffeomorphism types of smooth structures, distinguished by their genus--rank functions on $\widetilde X$. \qed
\end{thm}

\begin{examples}\label{bundles} (a) In \cite{GS} (following Theorem~9.4.29), the infinite connected sum of copies of $S^2\times S^2$ with a single end was given  as an example with infinitely many isotopy classes of smoothings, but which was not known to admit uncountably many smoothings. Theorem~\ref{cantor} distinguishes uncountably many diffeomorphism types, with considerable flexibility in the resulting genus filtration, and hence in the group of self-diffeomorphisms. See \cite{MinGen} for more details on the examples in this section.

\item[b)] Let $\{F_m\}$ be a countable, nonempty family of closed, connected surfaces (not necessarily orientable), and let $X$ be an end-sum of $\R^2$--bundles over these. When the family $\{F_m\}$ is finite, \cite{BE} produces (countably) infinitely many smoothings, and if $X$ is also orientable, there are uncountably many smoothings, cf.\ \cite[Corollary 9.4.25]{GS}, up to diffeomorphism in each case. These smoothings are made by end-sum with an exotic $\R^4$, which cannot increase minimal genera or obstruct orientation-preserving self-diffeomorphisms. In contrast, our method gives infinitely many diffeomorphism types with control on minimal genera, even when $\{F_m\}$ is infinite: One can construct a homeomorphism-invariant finite cover $\widetilde X$ of $X$ (with degree 1,2 or 4), with orientable domain, such that each lift of each $F_m$ is orientable \cite{MinGen}, then apply Theorem~\ref{inf}(d). For infinite $\{F_m\}$, such manifolds $X$ were posed in \cite{GS} (with each $F_m$ diffeomorphic to $\R P^2$) as examples not known to have exotic smooth structures. Applying Theorem~\ref{cantor} to $\widetilde X$, we obtain uncountably many diffeomorphism types of exotic smoothings on $X$, again with control of self-diffeomorphisms. For example, when everything is orientable, the subgroup $A_0\subset{\rm H}_2(X)$ generated by surfaces $F_m$ with a fixed upper bound on genus is preserved by continuous maps. We can arrange the diffeomorphism group to act transitively on infinite collections of such surfaces (when such collections exist with fixed genus and Euler number) while disallowing any permutations among homology classes of higher genus surfaces $F_m$. (Apply Lemma~\ref{main} relative to the infinite rank subgroup $A_0$.)

\item[c)] We can sometimes use Lemma~\ref{main} to study an open 4--manifold $X$ that is not a 2--handlebody interior. Theorem~\ref{MxR} does this when $X=M^3\times\R$ by embedding it in a 2--handlebody. For a different approach, suppose $L\subset X$ is a properly, tamely embedded 1--complex with image invariant (up to proper homotopy) under homeomorphisms of $X$. There is then a well-defined map from isotopy or diffeomorphism types of smoothings $\Sigma$ on $X$ to those on $X-L$, by first isotoping $\Sigma$ so that $L$ is smooth, then deleting $L$ (cf.\ \cite{CH}). It then suffices to distinguish the image smoothings in $X-L$ or note that the map preserves $G$. For example, if $X$ is a finite connected sum of 2--handlebody interiors (so ${\rm H}_3(X)\ne0$), we can take $X-L$ to be their end-sum, and Lemma~\ref{main} produces smoothings that are distinguished both on $X-L$ and $X$.
\end{examples}

When ${\rm H}_2(X)$ is finitely generated, we obtain complete control of the genus filtration.

\begin{thm}\label{filt}
Let $X$ be the interior of a connected 2--handlebody. Suppose that $X$ is oriented and ${\rm H}_2(X)$ is finitely generated. Then any filtration of ${\rm H}_2(X)$ consisting of $b_2(X)+1$ distinct direct summands is realized as the genus filtration of some Stein--Casson smoothing on $X$, and the characteristic genera can be chosen to increase arbitrarily rapidly.
\end{thm}

\begin{proof} Let $\alpha_1,\dots,\alpha_{b_2(X)}$ be a basis for ${\rm H}_2(X)$ such that the $i^{th}$ subgroup in the filtration is the span of  $\alpha_1,\dots,\alpha_{i-1}$. This basis is carried by some finite subhandlebody of $X$. Slide handles so that the first $b_2(X)$ 2--handles each represent the corresponding elements $\alpha_i$. Now apply Lemma~\ref{main} with $\pi=\id_X$.
\end{proof}

Having exploited lower bounds for the genus function, we now investigate how to control it more precisely. Our method can be applied in general via Legendrian Kirby diagrams, but we focus on a case avoiding this technology. Let $X$ be an orientable manifold that is an $\R^2$--bundle over a closed, connected surface $F$ of genus $g$. Let $\alpha$ be a generator of ${\rm H}_2(X)$ if $F$ is orientable, and otherwise a generator of ${\rm H}_2(\widetilde X)$ for the double cover $\widetilde X$ to which $F$ lifts orientably (which is a bundle over an orientable surface of genus $g-1$). We determine all possible values of $G(\alpha)$ when $F$ is orientable, with a similar statement in the nonorientable case.

\begin{thm}\label{bundle} For the bundle $X\to F$ as above,
\begin{itemize}
 \item[a)] If $F$ is orientable, then $X$ admits a smoothing with $G(\alpha)=g'$ if and only if $g'\ge g$.

 \item[b)] If $F$ is nonorientable, $g'\ge g$ and $g'\equiv g$ mod 2, then $X$ admits a smoothing for which $G(\alpha)=g'-1$. If $g'<g$, then no such smoothing exists.
\end{itemize}
These smoothings have Stein--Casson type for some orientation on $X$, unless $g'=g$ and $X$ is a bundle over $F=\R P^2$ with Euler number $e(X)=0$ or over $S^2$ with $|e(X)|\le 1$. There are such Stein--Casson smoothings for each orientation of $X$ whenever $|e(X)|\le -\chi(F)$, where $\chi$ denotes the Euler characteristic.
\end{thm}

\noindent The case of orientable $F$ is treated from a somewhat different viewpoint in \cite[Example~6.1(b)]{JSG}. For nonorientable $F$, one can similarly analyze the minimal genus of the generator of  ${\rm H}_2(X;\Z_2)$, for nonorientable surfaces whose $w_1$ pulls back from that of $F$, by reducing to the orientable case in $\widetilde X$ \cite{MinGen}.

\begin{proof} It suffices to realize the required smoothings, with the negative results following since a map from an orientable surface to one with larger genus must have degree 0. We start by building a model family of Stein surfaces homeomorphic to $\R^2$--bundles. The cotangent bundle $T^*F$ of $F$ has oriented total space and Euler number $e(T^*F)=-\chi(F)$. We can construct $T^*F$ as a Stein surface by complexifying the real-analytic manifold $F$, or by explicitly drawing a link diagram as in \cite{Ann}. Let $V_k$ be the Stein surface obtained by performing $k$ positive self-plumbings in the 2--handle of $T^*F$. The generator of ${\rm H}_2(V_k)$ (with $\Z_2$ coefficients if $F$ is nonorientable) is represented by a smoothly embedded surface $F_k$ diffeomorphic to $F\# kT^2$, obtained by smoothing the double points of the immersed copy of $F$. Then $F_k\cdot F_k=e(T^*F)+2k$, since a local model of each double point exhibits two intersections between a pair of copies of $F_k$. (This is basically the framing correction for the newly created kinky handle.) Let $U_{g,n,k}^\pm$ (where $n=e(T^*F)+2k$ and $+$ denotes the case of orientable $F$) be the Stein surface obtained from $V_k$ by adding $k$ Casson handles, with a single double point (positive) at each stage, to convert the kinky handle of $V_k$ into a Casson handle. This is homeomorphic to the $\R^2$--bundle over $F$ with Euler number $n$, and is defined for all $g,k\ge 0$ (except $g=0$ for $-$) and $n-2k=e(T^*F)=2g-2$ (for $+$) or $g-2$ (for $-$). We can then extend the notation to include Stein surfaces with all smaller values of the Euler number $n$ by removing the Casson handle and reattaching it with left twists added to the framing. When $F$ is orientable and $n$ is maximal, the adjunction inequality shows that $F_k\subset U_{g,n,k}^+$ has minimal genus in its homology class, and $c_1=0$. For smaller $n$, we control the signs of the rotation number corrections so that $F_k\cdot F_k+|\langle c_1,F_k\rangle|$ is independent of $n$, and $F_k\subset U_{g,n,k}^+$ still has minimal genus. In the nonorientable case, the orientable double covering of $F$ determines a double covering of $U_{g,n,k}^-$ by $U_{g-1,2n,2k}^+$, since the two lifts of the Casson handle can be combined into a single Casson handle with $2k$ positive double points at the first stage. The smooth surface $F_k$ in $U_{g,n,k}^-$ lifts to the corresponding orientable surface with minimal genus.

Now given $X$ as in the theorem, orient it to fix the sign of $e(X)$. Then $X$ is orientation-preserving homeomorphic to each $U_{g,n,k}^\pm$ for the given sign, with $g$ the genus of $F$ and $n=e(X)$. These are defined for all $k\ge 0$, provided that $e(X)\le 2g-2$ (resp. $g-2$). Furthermore, if $X$ is oriented so that $e(X)\le 0$, these are defined except when $k=0$ and $X$ is one of the specified exceptions. In the orientable case, $U_{g,n,k}^+$ pulls back to a Stein--Casson smoothing on $X$ with $G(\alpha)=g+k$, for each $k\ge 0$ except for the exceptional cases where $k=g=0$, realized by the standard smoothing. The theorem follows immediately for orientable $F$. When $F$ is nonorientable, $U_{g,n,k}^-$ similarly induces a smoothing, double covered by $U_{g-1,2n,2k}^+$.
\end{proof}

\section{Minimal genus at infinity}\label{Ends}

Before defining the genus function at infinity, we briefly review the theory of ends of manifolds with boundary, eg \cite{HR}. Informally, we explore the behavior of a topological manifold $X$ at infinity by considering the complements of successively larger compact subsets. More precisely, let $\{K_i|i\in\Z^+\}$ be an exhaustion of $X$ by compact subsets, meaning that $K_i\subset\int K_{i+1}$ for each $i$ and $X=\bigcup_{i=1}^\infty K_i$. Consider the {\em neighborhood system of infinity} $\{X-K_i\}$.

\begin{de} The {\em space of ends} of $X$ is $\E (X)=\limin \pi_0(X-K_i)$.
\end{de}

\noindent That is, an end $\epsilon\in\E (X)$ is given by a sequence $U_1\supset U_2\supset U_3\supset\cdots$, where each $U_i$ is a component of $X-K_i$. If we use a different exhaustion of $X$, the resulting space $\E (X)$ will be canonically equivalent to the original: The set is preserved when we pass to a subsequence, but any two exhaustions have interleaved subsequences. An equivalent definition of $\E(X)$ is as the set of equivalence classes of {\em rays}, proper maps $[0,\infty)\to X$, where we call two rays equivalent if their restrictions to $\Z^+$ are properly homotopic. A {\em neighborhood} of the end $\epsilon$ is an open subset of $X$ containing one of the subsets $U_i$. This notion allows us to topologize the set $X\cup\E(X)$ so that $X$ is homeomorphically embedded as a dense open subset and $\E(X)$ is totally disconnected \cite{Fr}. The resulting space is Hausdorff with a countable basis. When $X$ has only finitely many components, this space is compact, 
and called the {\em Freudenthal} or {\em end compactification} of $X$. In this case, $\E(X)$ is homeomorphic to a closed subset of a Cantor set. As a simple example, one can realize many homeomorphism types of end-spaces, from a single point to a Cantor set, starting with a fixed countable collection of closed manifolds (even just 2-spheres) and connected-summing them via various trees. A proper topological embedding $M\times[0,\infty)\to X$, for some closed, connected manifold $M$ with codimension 1 in $X$, determines an end of $X$, which is {\em topologically collared} by the embedding. Equivalently, a collared end is obtained from a manifold by removing a boundary component identified with $M$. Clearly, not all ends can be collared, but those that can are a good test case for our invariants.

The genus function at infinity has domain given by a naive attempt at defining homology at infinity. Given a manifold $X$ with an exhaustion $\{K_i\}$ by compact subsets, fix $k\in\Z^{\ge0}$ and consider the inverse limit ${\rm H}^\leftarrow_k(X)=\lim_\leftarrow {\rm H}_k(X-K_i)$ induced by inclusion. Each $\alpha\in{\rm H}^\leftarrow_k(X)$ is a sequence of elements $\alpha_i\in {\rm H}_k(X-K_i)$ that are mapped to each other by the corresponding inclusions. Clearly, ${\rm H}^\leftarrow_k(X)$ is independent of the choice of sequence of compact subsets. If an end of $X$ has a neighborhood collared by some $M$, then it determines a direct summand of ${\rm H}^\leftarrow_k(X)$ canonically isomorphic to ${\rm H}_k(M)$, and for different collared ends (possibly infinitely many), these summands are independent. When $X$ is a 4--manifold and $k=2$, the elements $\alpha_i$ representing a given $\alpha\in\Hinf(X)$ always have $\alpha_i\cdot\alpha_i=0$. This is because $\alpha_i\in {\rm H}_2(X-K_i)$ is represented by a surface contained in some $K_j$, but equals the image of $\alpha_j\in {\rm H}_2(X-K_j)$. We obtain more useful information from the genus function:

\begin{de} The {\em genus function at infinity} for a smooth 4--manifold $X$ is  the function $G_\infty\co\Hinf(X)\to \Z^{\ge 0}\cup\{\infty\}$  for which $G_\infty(\alpha)$ is the limit of the nondecreasing sequence of minimal genera $G(\alpha_i)$ in $X-K_i$.
\end{de}

\noindent That is, each $\alpha\in\Hinf(X)$ is represented by a sequence of homologous oriented surfaces avoiding successively larger compact subsets of $X$, and $G_\infty(\alpha)$ is the minimal possible limit of genera of such a sequence. This is clearly independent of the choice of exhaustion. We can now talk about the {\em genus--rank function at infinity} and {\em genus filtration at infinity} by analogy with Definition~\ref{genus filtration}, or discuss these for a single end. Unlike $G$, the function $G_\infty$ is subadditive, since we can add classes using disjoint representative surfaces.  In particular, the classes with finite $G_\infty$ form a subgroup.

\begin{Remarks}\label{H3} (a) Our naive end homology ${\rm H}^\leftarrow_k(X)$ is a quotient of the usual end homology, with kernel given by the derived limit $\lim^1_\leftarrow {\rm H}_{k+1}(X-\int K_i)$ \cite[Proposition~2.6]{L}; see also \cite{HR}.

\item[(b)] One can define an analog of $G_\infty$ using ${\rm H}^\leftarrow_3$ in place of $\Hinf$ and minimizing the first Betti numbers of the resulting sequences of 3--manifolds. This was discussed from a different viewpoint in \cite{BG} in the case where $X=R$ is an exotic $\R^4$, using the generator of ${\rm H}^\leftarrow_3(R)\cong\Z$. This {\em engulfing index}, denoted $e(R)$, is bounded below by Taylor's invariant. For large exotic $\R^4$'s (those with a compact subset $K$ not smoothly embedding in $S^4$) Taylor's invariant is frequently nontrivial, but it seems a good conjecture that  $e(R)$ is always infinite in this case (except possibly for punctured exotic 4--spheres). For small exotic $\R^4$'s (those that are not large), there are examples with $e(R)\le 1$ \cite{BG}, but Taylor's invariant always vanishes and there are no known lower bounds on $e(R)$.
\end{Remarks}

We first analyze $G_\infty$ on an end $\epsilon$ of $X$ topologically collared by a closed 3--manifold $M$. We can also realize this $M$ as the boundary of some compact 2--handlebody $H_0$. If $C\subset H_0$ is the core 2--complex of $H_0$, then  $H_0-C$ is identified with the domain $M\times[0,\infty)$ of the collar. This product structure gives a  canonically embedded infinite sequence $H_0\supset H_1 \supset H_2\supset\cdots$ of identical handlebodies with parallel boundaries, whose common intersection is $C$. The open sets $\int H_i-C$ then comprise a neighborhood system of the end of $H_0-C$, which is identified with $\epsilon$. Let $\pi\co \widetilde H_0\to H_0$ be a finite covering. In our applications, this will usually be the identity (so that tildes can be ignored) or the orientable double covering, but other coverings can also be useful. The lifted handlebodies $\widetilde H_i=\pi^{-1}(H_i)$ are nested, with common intersection $\widetilde C_i=\pi^{-1}(C_i)$. Now choose a nested sequence of subhandlebodies $H^s_0\subset H^s_1\subset H^s_2\subset\cdots\subset H_0$, with $H^s_0$ empty  and $H^s_1$ containing all of the 1--handles of $H_0$. In the upcoming proof, we will take $H^s_i$ to be a subhandlebody of $H_i$ rather than $H_0$, but for our present homological discussion, we can think of all handlebodies $H_i$ as being the same. Let $\widetilde H^s_i=\pi^{-1}(H^s_i)$, and let $B_i\subset {\rm H}_2(\widetilde H_0)$ denote the image under inclusion of ${\rm H}_2(\widetilde H^s_i-\widetilde C)$. This is the span in ${\rm H}_2(\widetilde H_0)$ of all surfaces in $\partial\widetilde H_i=\pi^{-1}(M)$ that lie in $\widetilde H^s_i$. Suppose inclusion induces an injection $\iota\co{\rm H}_2(\partial\widetilde H_0)\to{\rm H}_2(\widetilde H_0)$. Then we can identify the groups $B_i$ as nested subgroups of ${\rm H}_2(\partial\widetilde H_0)=\Hinf(\widetilde H_0-\widetilde C)$. In particular, when $\pi=\id_{H_0}$, we have identified a filtration of the summand of $\Hinf(X)$ coming from $\epsilon$, and the end of any finite cover of $X$ is similarly accessible. If the manifold made by gluing together $X$ and $H_0$ admits a smoothing, our main Lemma~\ref{nest} modifies its restriction to $X$ so that the filtration $\{B_i\}$ becomes part of the genus filtration at infinity of $\epsilon$ (or its cover). This modified smoothing has a special form that will be useful: It is induced by a topological isotopy of $H_0-C$ rel boundary.

\begin{lem} \label{nest} Let $H_0$, $\pi$ and $\{ H^s_i\}$ be as above, with $\widetilde H_0$ oriented and $\iota$ injective. Then for any smooth structure $\Sigma$ on $\int H_0$, there is an (arbitrarily rapidly) increasing sequence $\{k_i\thinspace |\thinspace i>0\}$ of integers and a topological isotopy $\varphi_t$ of the inclusion map $\varphi_0\co H_0-C\to H_0$, rel a neighborhood of $\partial H_0$, so that in the pulled back smooth structure $\varphi_1^*\Sigma$, each $B_i$ with $i>0$ is the span in $\Hinf(\widetilde H_0-\widetilde C)$ of all classes $\alpha$ with $G_\infty(\alpha)\le k_i$.
\end{lem}

 \noindent Since $H_0$ is compact, it has only finitely many subhandlebodies. However, we take the index set of $\{ H^s_i\}$ to be $\Z^{\ge0}$. The filtration stabilizes at some $B_n\subset\Hinf(\widetilde H_0-\widetilde C)$ (whose rank or corank could be 0). The present purpose of the terms $k_i$ with $i\ge n$ is to imply that $B_n$ is the subgroup of all classes with finite $G_\infty$.

\begin{proof} To construct our isotopy, we will need Quinn's Handle Straightening Theorem  \cite[2.2.2]{Q}; see  \cite[Section~5]{JSG} for an exposition. Suppose $f\co D^k\times\R^{4-k}\to W$ is a homeomorphism from an open $k$--handle to a smooth manifold, restricting to a diffeomorphism of the boundaries. If $k=0,1$, Quinn's theorem states that $f$ is topologically isotopic, rel boundary and a neighborhood of the end, to a homeomorphism that is a local diffeomorphism in a neighborhood of the core $D^k\times\{0\}$. If $k=2$ this fails, but after such an isotopy we can assume, by \cite[Proposition~2.2.4]{Q} strengthened as in \cite[Theorem~5.2]{JSG}, that $f(D^2\times D^2)$ lies in a smoothly embedded Casson handle $CH\subset W$, and that $f|D^2\times D^2$ extends to a homeomorphism $D^2\times \R^2\to CH$. The Casson handle $CH$ will typically have many double points of both signs at each stage, but the construction is sufficiently flexible that we can replace the given $CH$ by any refinement of it without affecting the discussion. For example, if $f$ is initially the identity map of an open 2--handle, then any standardly embedded Casson handle fits into such a description. In general, we can obtain one further property.  We think of $CH$ as being obtained from its first stage $T$ by adding Casson handles comprising the higher stages and removing excess boundary. Since $T$ is a smooth regular neighborhood of an immersed disk, we can smoothly isotope it to a smaller neighborhood $T'\subset T$ of the disk, intersecting $\partial T$ only in its attaching region. Since each higher-stage Casson handle is itself homeomorphic to $D^2\times \R^2$, it contains a canonically embedded topological 2--handle $D^2\times D^2$ that, after a smooth isotopy, attaches to $T'$ inside $CH$. Then the union $h$ of $T'$ with these new 2--handles is itself a 2--handle inside $CH$, and $CH-\int h$ is a collar of the outer boundary of $h$, homeomorphic to $D^2\times (\R^2-\int D^2)$. Thus, $h$ is unknotted in $CH$, so it can be assumed to equal $f(D^2\times D^2)$. In conclusion, any $f\co D^2\times\R^2\to W$ as above is topologically isotopic, rel boundary and a neighborhood of the end (so ambiently), to a new homeomorphism sending $D^2\times D^2$ to a topological 2--handle $h$ inside a smoothly embedded Casson handle $CH\subset W$ that is an arbitrary refinement of some fixed one, and $h$ is obtained by smoothly shrinking the first stage of $CH$ and then adding topological 2--handles to it.

\begin{figure}
\labellist
\small\hair 2pt
\pinlabel $F$ at 23 90
\pinlabel $F'$ at 220 90
\pinlabel $h_1$ at 110 35
\pinlabel $h_2$ at 87 6
\pinlabel (a) at -2 4
\pinlabel 0 at 93 137
\pinlabel {$\widetilde H_i$} at -2 135
\pinlabel $h_1$ at 307 35
\pinlabel $h_2$ at 284 6
\pinlabel (b) at 195 4
\pinlabel 0 at 290 137
\pinlabel $CH$ at 358 83
\pinlabel 0 at 356 102
\endlabellist
\centering
\includegraphics{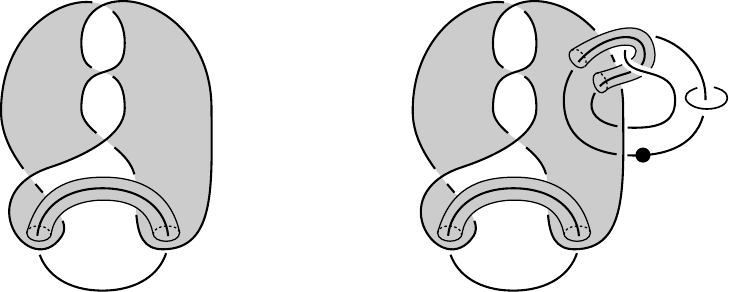}
\caption{The closed surface $F$ in (a) goes over the handle $h_1$ comprising $\widetilde H_i^s$ but must be disjoint from $h_2$ (which intersects $\widetilde C$). It is homologous in $\partial\widetilde H_i^s-\widetilde C$ to the surface $F'$ seen in (b) that lies in $\partial\widetilde T_i$ (realized by $h_1$ and the dotted circle).}
\label{fig}
\end{figure}
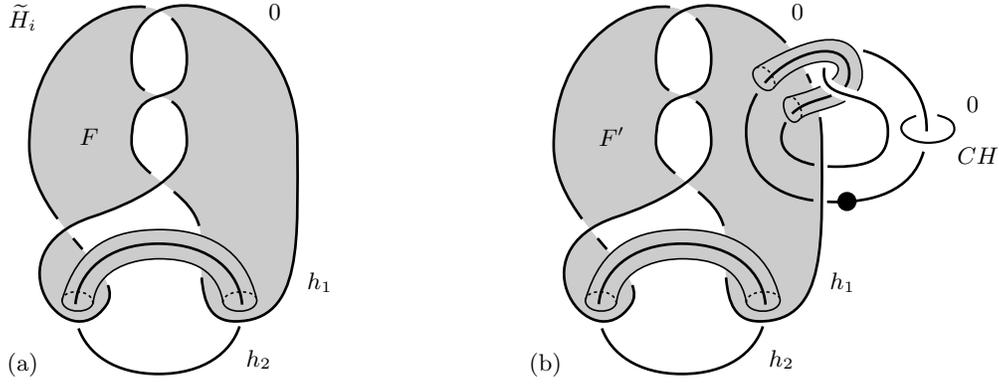

We will also need to locate smooth surfaces representing elements $\alpha\in B_i$ for each $i$. Any such element is carried by $\partial\widetilde H_i^s-\widetilde C$. Suppose that the smoothing on $\int H_i$ has Casson type. Let $T_i$ denote the manifold obtained from $H_i^s$ by replacing its 2--handles with the first-stage towers of the corresponding Casson handles. We can then recover $\int H_i^s$ with its Casson-type smoothing by adding the higher-stage Casson handles onto $T_i$ and removing the remaining boundary, and the same holds for the lift $\widetilde T_i$ to $\widetilde H_i^s$.  The class $\alpha$ is represented by some closed surface $F$ in $\partial\widetilde H_i^s$ that is disjoint from $\widetilde C$. Figure~\ref{fig}(a) shows an example with $\widetilde H_i$ made from 2--handles $h_1$ and $h_2$, where only $h_1$ lies in $\widetilde H_i^s$. The surface necessarily avoids $h_2$, whose core lies in $\widetilde C$. We obtain $\widetilde T_i$ from $\widetilde H_i^s$ by adding dotted circles that are Whitehead doubles of meridians of attaching circles. These puncture $F$, but only in algebraically canceling pairs, so the punctures can be repaired as in Figure~\ref{fig}(b). We obtain a surface $F'$ that is homologous to $F$ in $\partial\widetilde H_i^s-\widetilde C$. Since each higher-stage Casson handle $CH$ is attached to a meridian of a new dotted circle, it can be assumed to be disjoint from $F'$. Thus, $F'$ lies in $\partial\widetilde T_i$  avoiding both $\widetilde C$ and the higher-stage Casson handles, and represents $\alpha$ as required.

We construct the required isotopy $\varphi_t$ by perturbing the identity on $H_0$ to smooth the 0-- and 1--handles of $H_1$, then applying induction on the handlebodies $H_i$. For each $i>0$, we wish to ambiently isotope $H_i$ inside $H_{i-1}$ (or more precisely, inside the image of $H_{i-1}$ under the previous isotopy) so that $H_i$ becomes embedded in a larger copy $U_i$ of $\int H_i$ whose smooth structure inherited from $\Sigma$ as an open subset of $H_0$ arises from Lemma~\ref{main} for the subhandlebodies $H_1^s\subset H_2^s\subset\cdots\subset H_{i-1}^s\subset H_i$ and some sequence $\{k_j|1\le j<i\}$, with the required generating surfaces for each $A_j$ lying in $\widetilde T_j$. We also require each 2--handle of $H_i$ to sit inside the corresponding Casson handle as in the first paragraph of the proof. As an induction hypothesis, we assume the copy of $H_{i-1}^s$ in $H_i$ has already been isotoped into a suitable $U_{i-1}^s$ in this manner, as is vacuously true when $i=1$. Choose a finite generating set for $B_{i-1}$, represented by surfaces $F_r$ in $\partial\widetilde T_{i-1}\subset\widetilde H_{i-1}^s$ as constructed in the previous paragraph. The Handle Straightening Theorem isotopes $H_i$ rel $H_{i-1}^s$ to make its 2--handles suitably embedded in Casson handles so that $U_{i-1}^s$ extends to a suitable $U_i$ containing $H_i$. Since the new Casson handles could have been refined arbitrarily during the construction, we can choose them to arise from Lemma~\ref{main} for the sequence $H_1^s\subset\cdots\subset H_{i-1}^s\subset H_i$, using each previous $k_j$ (defined for $j<i-1$), and choosing $k_{i-1}$ larger than the genus of each $F_r$ and generating surface in $\widetilde T_{i-1}$ for $A_{i-1}$. (Addendum~\ref{mainadd} guarantees that we can do this without disturbing $U_{i-1}^s$, refining the originally embedded Casson-type smoothing. When $i=1$, this step merely achieves the first sentence of that addendum for $H_1$.) To restore the induction hypothesis, note that each 2--handle of $H_i^s\subset H_i$ is built from a copy of the first stage of its containing Casson handle by adding topological 2--handles. The smooth handles and kinky handles of $H_i^s$ comprise an embedding of $T_i$. Smoothly squeezing $T_i$ into its interior, and topologically squeezing the topological 2--handles, yields a subset that can be identified with the copy of $H_i^s$ in $H_{i+1}$ after we topologically isotope the latter. This version of $H_i^s$ is embedded in $U_i^s$ (obtained from $U_i$ by removing Casson handles) as required for the induction, with each $k_j$ continuing to work as required in $U_i^s$ since each $T_j$ was squeezed smoothly and inclusion induces an injection ${\rm H}_2(\widetilde H_i^s)\to {\rm H}_2(\widetilde H_i)$.

To complete the proof, note that we now have the required increasing sequence $\{k_i\}$, and a sequence of topological ambient isotopies that together comprise a family of homeomorphisms of $H_0$ rel a neighborhood of $\partial H_0$, parametrized by $[0,1)$ and beginning with $\id_{H_0}$. Every point in $H_0-C$ lies outside some $H_i$, so has a neighborhood on which the family is independent of $t$ sufficiently close to 1. Thus, there is an induced continuous family $\varphi_t\co H_0-C\to H_0$, for $0\le t\le1$, that is easily seen to be a topological isotopy (not ambient) of the inclusion map, fixing a neighborhood of $\partial H_0$. For each $i>0$, we exhibited $\varphi_1^*\Sigma$--smooth surfaces of genus at most $k_i$ in $\partial\widetilde T_i-\widetilde C$ generating $B_i$. Since the induction smoothly squeezes $\widetilde T_i$ into each subsequent stage, these surfaces have smoothly isotopic copies in each subspace $\widetilde H_m-\widetilde C$, $m\ge i$, determining classes in $\Hinf(\widetilde H_0-\widetilde C)$ with $G_\infty\le k_i$ generating $B_i$. To see that every $\alpha\in\Hinf(\widetilde H_0-\widetilde C)$ with $G_\infty(\alpha)\le k_i$ lies in $B_i$, note that $\alpha\cdot\alpha=0$, so the genus constraint implies that $\alpha$ maps to the subgroup $A_i={\rm H}_2(\widetilde H_i^s)$ of ${\rm H}_2(\widetilde H_0)$, as well as to the subgroup ${\rm H}_2(\partial\widetilde H_i)$ (by definition). The union of these two subspaces $\widetilde H_i^s$ and $\partial\widetilde H_i$ lies ${\rm H}_2$--injectively in $\widetilde H_0$ (since all 1--handles of $\widetilde H_i$ lie in $\widetilde H_i^s$), so the Mayer--Vietoris sequence of the pair shows that $\alpha$ pulls back to the intersection. Thus, $\alpha$ lies in $B_i$ as required.
\end{proof}

\begin{Remark}\label{nestadd} The proof gives more detailed information about the smoothing of the end: Each class in ${\rm H}_2(\widetilde H_i^s-\widetilde  C)$ is represented by a surface $F$ in $\partial\widetilde H_i^s-\widetilde C$ that extends via the canonical product structure on $\widetilde H_i^s-\widetilde C$ to a $\varphi_1^*\Sigma$--smooth, proper embedding $F\times [0,\infty)\to\widetilde H_i^s-\widetilde C$. These surfaces can be chosen before $k_i$ is defined.
\end{Remark}

While Lemma~\ref{nest} is powerful enough for most of our applications, it is also highly restricted, since more typical ends are not collarable. However, the above proof works in much more generality. Instead of requiring the handlebodies $H_i$ to be canonically nested copies of $H_0$, we can take any infinite topological nest of nonempty, compact 2--handlebodies intersecting in a compactum $C$, and use this to analyze manifolds with proper topological embeddings of $H_0-C$. Any such nest realizes a neighborhood of infinity in some manifold without boundary, for example, the double of $H_0$ with one copy of $C$ removed. We do not need the 2--handlebodies to be connected, so the end space may be a Cantor set. Similarly, the group $\Hinf (H_0-C)$ need not be finitely generated. The proof of the lemma requires the subhandlebodies $H_i^s\subset H_i$ to be stable in the sense that each $H_i^s$ has a subhandlebody that is a canonically embedded copy of $H_{i-1}^s$, and the remaining handles of $H_i$ respect the resulting product structure $\partial H_{i-1}^s\times I$. We introduce a finite covering $\pi$ as before with $\widetilde H_0$ oriented, and replace our map $\iota$ by the inclusion--induced map
 $$\iota_\infty\co \Hinf(\widetilde H_0-\widetilde C)=\limin {\rm H}_2(\widetilde H_i-\widetilde C)\to\limin {\rm H}_2(\widetilde H_i).$$
This no longer need be injective, provided that we work in $\im\iota_\infty$, or equivalently, in $\Hinf(\widetilde H_0-\widetilde C)$ modulo $\ker\iota_\infty$. We let $B_i=\iota_\infty({\rm H}_2(\widetilde H_i^s-\widetilde C))$. (These need not stabilize as in the collared case.) The requirement that all 1--handles lie in $H_1^s$ can be replaced by the weaker notion of {\em controlled instability}, the ${\rm H}_2$--injectivity condition asserting that the final Mayer--Vietoris argument works. We obtain the following lemma \cite{MinGen}:

\begin{lem} \label{nest2} For any smooth structure $\Sigma$ on $\int H_0$, there is an (arbitrarily rapidly) increasing sequence $\{k_i\thinspace |\thinspace i>0\}$ of integers and a topological isotopy $\varphi_t$ of the inclusion map $\varphi_0\co H_0-C\to H_0$, rel a neighborhood of $\partial H_0$, so that in the pulled back smooth structure $\varphi_1^*\Sigma$, each $B_i$ with $i>0$ is the span in $\im\iota_\infty$ of all classes $\alpha \in \Hinf(\widetilde H_0-\widetilde C)$ with $G_\infty(\alpha)\le k_i$.
\end{lem}

\noindent The proof is essentially the same as before, with added care surrounding the Mayer--Vietoris argument. (In Figure~\ref{fig}, $\widetilde C$ might now appear as, for example, a Bing continuum in the attaching region of $h_2$, with $F$ wrapped through it, but the proof still works.) A high-dimensional argument shows that the resulting smoothing  $\varphi_1^*\Sigma$ always lies in the stable isotopy class obtained by restricting the unique one on $H_0$.

In principle, Lemma~\ref{nest2} is much more powerful than Lemma~\ref{nest}, although examples for which it is required are necessarily somewhat complicated. The extra power of this lemma seems useful for attacking Question~\ref{Q} in full generality. We apply it to infinite 1--handlebody interiors in Theorem~\ref{1h}(b), to more general manifolds in Theorem~\ref{endsystemdiff}, and to collared ends with $\iota$ not injective in Theorem~\ref{M3iso}.

\begin{example}\label{mess}
Let $H_0$ be a handlebody on a framed link $L\subset\partial B^4$ whose linking pairing vanishes on one component $K$. Modify $L$ by leaving $K$ alone but using the satellite construction to insert a topologically slice link into a tubular neighborhood of each other component, respecting the framing. (For example, connected summing with topologically slice knots, Whitehead and Bing doubling, $(n,1)$--cabling and their ramified versions are special cases for the 0--framing.) The resulting handlebody $H_1$ topologically embeds in $H_0$ with a stable 2--handle on $K$. Continue by induction to get an infinite nest of 2--handlebodies, with $H^s_i$ the handlebody on $K$ for $i>1$, but empty for smaller $i$. Since there are no 1--handles, each inclusion $H_i-H_j\subset H_i$ is ${\rm H}_2$--injective, as is $\iota_\infty$, so $B_2\ne0$ but $B_1=0$. Thus, the lemma gives infinitely many diffeomorphism types of smoothings on $H_0-C$ (cf.\ Example~\ref{mess2}).
\end{example}

\section{Applications of minimal genera at infinity}\label{Apps}

We first consider a connected topological 4--manifold $X$ (possibly with boundary) with an end $\epsilon$ collared by a closed, connected 3--manifold $M$. We take the associated proper embedding to be inclusion of a closed subset $M\times[0,\infty)\subset X$. Let $X^*=X-M\times(1,\infty)$ be the result of replacing the end by a boundary component. We assume its Kirby--Siebenmann invariant $ks(X^*)$ vanishes. (This is automatic unless $X^*$ is compact.) Recall that ${\rm H}_2(M)$ is a direct summand of $\Hinf(X)$. We show that any filtration of  ${\rm H}_2(M)$ by direct summands lies in the genus filtration of that end for some smoothing of $X$, provided that $M$ is orientable. Otherwise, a similar statement holds for the orientable double cover $\widetilde M$ of $M$.

\begin{thm}\label{collar}
Suppose $X$ has a collared end as above, with $ks(X^*)=0$ if $X^*$ is compact. If $M$ is orientable, let $C_1\subset C_2\subset\cdots\subset C_n\subset {\rm H}_2(M)$ be a filtration by direct summands. Then there is a smoothing of $X$ and an arbitrarily rapidly increasing sequence $\{k_i\thinspace |\thinspace 1\le i<n\}$ such that $C_i$ is the span of all classes in ${\rm H}_2(M)$ with finite minimal genus at infinity (if $i=n$) or minimal genus at infinity at most $k_i$ (if $i<n$). If  ${\rm H}_2(M)\ne0$, then $X$ has infinitely many diffeomorphism types of smoothings distinguished by $G_\infty$. If $M$ is nonorientable, the same holds for any $\Z_2$--invariant filtration of ${\rm H}_2(\widetilde M)$ by direct summands  (where the smoothings of $X$ are distinguished by $G_\infty$ in the orientable double cover of $X$).
\end{thm}

\noindent It follows that $C_n$ is the set of all classes in ${\rm H}_2(M)$ with finite $G_\infty$. The hypotheses allow $C_n$ to be 0 or all of ${\rm H}_2(M)$.

\begin{proof} We begin with the case when $M$ is orientable and $X^*$ is smoothable. To exhibit $M$ as the boundary of a suitable handlebody, first choose a basis $\{\alpha_r\}$ for ${\rm H}_2(M)$ for which each $C_i$ is the span of some subcollection. For the corresponding dual basis in ${\rm H}_1(M)$ mod torsion, represent each element by an embedded circle. Add a 2--handle to $M\times I$ along each of these circles to obtain a cobordism $V$ from $M$ to a new orientable 3--manifold. This, in turn, bounds a compact 4--manifold $W$ consisting of a 0--handle and 2--handles. The handlebody $H_0=W\cup V$ bounded by $M$ also consists of a 0--handle and 2--handles. To apply Lemma~\ref{nest}, note that the inclusion--induced map $\iota\co {\rm H}_2(M)\to {\rm H}_2(H_0)$ is injective, since $H_0$ has no 1--handles and so is built from its boundary $M$ by adding handles of index $\ne 3$. For $1\le i\le n$, let $H_i^s$ be the subhandlebody of $H_0$ obtained from $W$ together with the 2--handles of $V$ constructed (upside down) from the duals of the basis for $C_i$. Then the subset of $\{\alpha_r\}$  carried by $H_i^s-C$ is precisely our basis for $C_i$, so $B_i=C_i$. (Every such basis element clearly lies in $B_i$. Conversely, every class in ${\rm H}_2(M)$ is uniquely a linear combination of classes $\alpha_r$. By injectivity of $\iota$, if any $\alpha_r\notin C_i$ appears, the class cannot lie in $B_i$.) For $i>n$, let $H_i^s=H_n^s$ so $B_i=C_n$. Lemma~\ref{nest}, applied to the standard smoothing of $\int H_0$, gives smoothings of $M\times(0,\infty)$ that we can assume are standard on $M\times(0,2)$. These fit together with the given smoothing on $X^*$ (by unique smoothing of 3--manifolds) to give the required smoothings of $X$.

When  ${\rm H}_2(M)\ne0$, it is easy to distinguish infinitely many diffeomorphism types if $X$ has only one end: Take $C_1=0$ but $C_2\ne 0$, and vary $k_1$. Otherwise, we need to rule out diffeomorphisms sending $\epsilon$ to a different end. If there is an end that is not collared, then we can we can smooth $\epsilon$ as above with $C_2\ne 0$, smooth all remaining $M$--collared ends as above with $C_n=0$, and then extend over the remaining noncompact manifold. Any diffeomorphism between two of the resulting smoothings must then preserve $\epsilon$, so $k_1$ distinguishes the smoothings as before. If all ends are collared, then the end space $\E(X)$ is discrete and hence finite. Fix an arbitrary smoothing on the noncompact manifold $X^*$, and extend this over $X$ as before. Each end of $X^*$ has a genus (possibly infinite) minimized over its nonzero homology classes. When the smoothings on $\epsilon$ are chosen so that $k_1$ exceeds all such finite minimal genera, then any diffeomorphism again preserves $\epsilon$ so that infinitely many smoothings are again distinguished on $X$.

We now sketch the remaining cases. (See \cite{MinGen} for details.) The proof for a nonorientable $M$ is similar, with $V$ constructed from $M\times I$ using pushed down circles arising from the filtration of ${\rm H}_2(\widetilde M)$. Since the handlebody $W$ capping the other boundary of $V$ is nonorientable, it must contain a 1--handle, so more work is required to arrange $\iota$ to be injective. Then $\Z_2$--invariance of the filtration guarantees that $B_i=C_i$, and the rest follows as before. In the final case, $X^*$ is unsmoothable, so it must be compact with $ks(X^*)=0$. Thus, by standard smoothing theory (see \cite{FQ}), we can smooth the connected sum $X^*\# mS^2\times S^2$ for sufficiently large $m$. For $H_0$ as before, the compact manifold $Y=(X^*\# mS^2\times S^2)\cup_\partial H_0$ then inherits a smooth structure that is standard on $H_0$. Each $S^2\times S^2$ summand determines a topologically embedded handlebody $h_r$ ($r=1,\dots,m$) in $Y$ with boundary $S^3$, consisting of a 0--handle and two 2--handles. Let $H'_0$ be the handlebody obtained from $H_0$ by ambiently attaching each $h_r$ along a 1--handle, and let $C'$ be its core. Then $Y-C'$ is homeomorphic to $X$, so it suffices to apply Lemma~\ref{nest} to $\int H'_0$, with smoothing $\Sigma$ inherited from $Y$ (exotic on each $\int h_r$), and all of the subhandlebodies enlarged to contain each $h_r$.
\end{proof}

\begin{Remark} This family of smoothings lies in a single stable isotopy class. We can choose this to be any stable class restricting to the standard one on each collar to which we applied Lemma~\ref{nest}. We can arrange to only use one collar unless every collared end of $X$ is homeomorphic to infinitely many others.
\end{Remark}

We can sometimes combine constraints on $G_\infty$ with those of the ordinary genus function $G$:

\begin{thm}\label{MxR}
 Let $M$ be a closed, connected 3--manifold, and let $\widetilde M$ denote $M$ (if orientable) or its orientable double cover. Let $\{C_i\}$ and $\{D_j\}$ be filtrations of ${\rm H}_2(\widetilde M)$ by direct summands, where the largest $C_i$ is allowed to be a proper summand, and both filtrations are required to be $\Z_2$--invariant in the nonorientable case. Then there is a smoothing of $M\times \R$ for which each $D_j$ is in the  genus filtration of $\widetilde M\times \R$, and $\{C_i\}$ is contained in the genus filtration of one end as in Theorem~\ref{collar}. The corresponding sequences of integers can be chosen to increase arbitrarily rapidly, with those for $\{D_j\}$ chosen first.
\end{thm}

Of course, the genus--rank function of $\widetilde M\times \R$ is a diffeomorphism invariant of the smooth structure on $M\times \R$, and the corresponding genus--rank function of one end is invariant under end-preserving diffeomorphisms. For example, we recover the result of Bi\v zaca and Etnyre \cite{BE} that every $M\times \R$ has infinitely many diffeomorphism types of smoothings, under the additional hypothesis that $b_2(\widetilde M)\ne 0$ (which could be relaxed by considering other covers, cf.\ Example~\ref{covers}). The smoothings constructed here are quite different from those of \cite{BE}, which are obtained by end-summing with an exotic $\R^4$, so contain a smoothly embedded copy of $M$ and have $G$ and $G_\infty$ bounded above by those of the standard smoothing (cf.\ Theorem~\ref{uncountable}). Many 3--manifolds have the property that every self-homotopy equivalence is homotopic to the identity. For such $M$, the entire genus function of $\widetilde M\times \R$ is a diffeomorphism invariant of the smooth structure on $M\times \R$, and the corresponding genus function at infinity is invariant under end-preserving diffeomorphisms.

\begin{proof} Construct a handlebody $H$ bounded by $M$ as in the previous proof, using the filtration $\{D_j\}$ (so $B_j=D_j$). The proof of Lemma~\ref{nest} (at a finite stage of the induction) gives a smooth structure $\Sigma_0$ on $\int H$ and an arbitrarily rapidly increasing finite sequence $\{k_j\}$ such that each $D_j\subset {\rm H}_2(\widetilde M)\subset {\rm H}_2(\widetilde H)$ is the span of all smooth surfaces in $\int\widetilde H$ with genus at most $ k_j$ and representing classes in ${\rm H}_2(\widetilde M)$: We constructed a finite spanning set of such surfaces $F_r$ in the complement of the core $\widetilde C$, and any such surface lies in $D_j$ by the final Mayer--Vietoris argument. Let $H'\subset H$ be a canonically embedded handlebody in a neighborhood of $C$ disjoint from the image of each $F_r$ for each $D_j$. We now adjust the end of $H-C$ inside $H'$: After sliding 2--handles of $H'$, we can assume it was constructed as in the previous proof, but using the filtration $\{C_i\}$. Lemma~\ref{nest} gives an isotopy $\varphi_t\co H-C\to H$ that is the identity outside $H'$, such that  $\{C_i\}$ behaves as required for the smooth structure $\varphi_1^*\Sigma_0$. Since the smooth structure has not been changed outside $H'$, each $D_j$ is still spanned by the previous surfaces $F_r$ of genus at most $ k_j$ in $\int\widetilde  H-\widetilde C=\widetilde M\times\R$. But in this new smoothing, every surface with genus at most $ k_j$ still pulls back from $\int\widetilde H$, representing a class in ${\rm H}_2(\widetilde M)$, so this class lies in $D_j$ (by our choice of $\Sigma_0$). Thus, $D_j$ is the span of all  $\pi^*\varphi_1^*\Sigma_0$--smooth surfaces with genus at most $ k_j$ in $\widetilde M\times\R$, as required.
\end{proof}

\begin{Remark}\label{MxRsum} This idea can be used whenever $X$, with its chosen ends capped by 2--handlebodies, embeds into some 2--handlebody. We can apply Lemma~\ref{main} to control the genus function of $X$, and separately control the genus functions of collared ends as above. As another variation, we can specify the genus filtration of each end of $M\times\R$ separately, while leaving $M\times\{0\}$ smoothly embedded.
\end{Remark}

Next we illustrate the reach of $G_\infty$ with some examples. We first consider 1--handlebody interiors. While the 2--homology, and hence the genus function, of any cover of these are trivial, we can distinguish infinitely many smoothings by the genus function at infinity. In the simplest case, we realize all possible values of $G_\infty$ on a generator. At the opposite extreme, infinite 1--handlebody interiors have noncollarable ends, but are accessible by Lemma~\ref{nest2}. The subsequent Example~\ref{covers} shows the utility of using other covers of an end, even if they are not defined on all of the 4--manifold.

\begin{thm}[a]\label{1h} Let $\alpha\in \Hinf(S^1\times\R^3)\cong\Z$ be a generator. Then every $k\in \Z^{\ge0}\cup\{\infty\}$ is realized as $G_\infty(\alpha)$ for some smooth structure on $S^1\times\R^3$.

\item[b)] For any 1--handlebody interior $X$ with ${\rm H}_1(X)\ne 0$, there are infinitely many diffeomorphism types of smooth structures distinguished by $G_\infty$ (on the orientable double cover if relevant). If ${\rm H}_1(X)$ is not finitely generated, then there are uncountably many.
\end{thm}

\begin{proof} For (a), embed $S^1\times\R^3$ in $S^4$ as the complement of the standard $S^2$. The latter is an intersection of canonically nested handlebodies $H_i$ diffeomorphic to $S^2\times D^2$, each consisting of a 0--handle and 2--handle. The inclusion $\iota$ is an isomorphism, so letting each $H_i^s$ be empty in Lemma~\ref{nest} immediately realizes $k=\infty$. For finite $k>0$, we let $H_i^s$ be $H_i$ for each $i>0$, and sharpen the proof of Lemma~\ref{nest}. Let $U_1=U^+_{0,0,k}$ as in the proof Theorem~\ref{bundle} be standardly embedded in $H_0$. As in the proof of the lemma, we can assume the 2--handle of each $H_i$ is made by attaching topological 2--handles to a narrower version of the first stage of $U_1$, which has $k$ double points. By construction, the surface $F'$ from the proof of Lemma~\ref{nest}, representing $\alpha$ in ${\rm H}_2(H_i-S^2)$, has genus $k$ for each $i$. Since this is the minimal genus of the generator of $U^+_{0,0,k}$, the resulting topological isotopy completes the proof of (a).

For (b), it is routine to reduce to the connected case. Then for ${\rm H}_1(X)$ finitely generated, Theorem~\ref{collar} completes the proof. Otherwise, we must first find a suitable nest of 2--handlebodies realizing the end. Write $X$ as a nested union of compact, connected 1--handlebodies, each canonically embedded in a subhandlebody of the next. We can then inductively embed $X$ in $S^4$ or (if nonorientable) in $\R P^4$, by attaching the 1--handles of $X$ ambiently. Each complementary region $H_i$ is diffeomorphic to a boundary sum of copies of $S^2\times D^2$, and one disk bundle over $\R P^2$ in the nonorientable case. These admit 2--handlebody structures, and by arranging the successive 1--handles from $X$ to avoid chosen 2--handles of $H_i$, we can construct subhandlebodies $H^s_i$ so that ${\rm H}_2(H_i^s)$ increases without bound. Lemma~\ref{nest2} now controls infinitely many characteristic genera at infinity, so the result follows by careful bookkeeping.
\end{proof}

By construction, these smooth manifolds all embed as open subsets of $S^4$ (if orientable) or $\R P^4$, so have vanishing Taylor invariant. As is typical, we have great flexibility in controlling the genus filtration and genus--rank function at infinity for the above smoothings of $X$, as well as their diffeomorphism groups. Note that Part (a) generalizes to the complement in any smooth, closed 4--manifold of a smooth 2--sphere with trivial normal bundle.

\begin{example}\label{covers} It can be useful to consider covers other than an orientable double cover. For example, let $H_0$ be orientable and made from a single handle of each index 0, 1 and 2, with $\pi_1(H_0)\cong \Z_n$ for $n>1$. Let $\widetilde H_0$ be its universal cover, which is homotopy equivalent to a wedge of $n-1$ spheres. Then $\partial H_0$ is a rational homology 3--sphere, but $\partial \widetilde H_0$ can be arranged not to be, with $\iota$ injective. (In the simplest case, $\widetilde H_0$ is a boundary sum of $n-1$ copies of $S^2\times D^2$.) We know $\partial H_0$ also bounds a smooth, simply connected 4--manifold, whose interior $X$ has trivial $G_\infty$ and no nontrivial covers. However, Lemma~\ref{nest} still gives infinitely many nondiffeomorphic smoothings of $X$, distinguished by $G_\infty$ on the cover $\partial \widetilde H_0\times \R$ of a neighborhood of infinity in $X$. (Note that $\partial H_0$ has only finitely many $\Z_n$--covers.)
\end{example}

Theorem~\ref{collar} can be extended to more general ends by using Lemma~\ref{nest2}. We again start with a topological 4--manifold $X$, but this time assume there is a proper embedding $H_0-C\to X$ for a nest of 2--handlebodies as in Lemma~\ref{nest2}. This picks out a nonempty subset of the end space $\E(X)$ (that is open, closed and possibly uncountable) to take the place of the collared end $\epsilon$ in Theorem~\ref{collar}. Fix a finite covering $\pi\co \widetilde H_0\to H_0$ with $\widetilde H_0$ orientable. If this happens to extend over $X$, as for the identity or orientable double cover, then we can identify $\Hinf(\widetilde H_0-\widetilde C)$ with a direct summand of $\Hinf(\widetilde X)$. Suppose we have a stable collection of subhandlebodies $H_i^s\subset H_i$ with controlled instability for $\pi$, and that $ks(X-\int H_0)=0$. Then there is a smoothing on $X$ and an arbitrarily rapidly increasing sequence $\{k_i\}$ for which each $B_i$ satisfies the conclusion of Lemma~\ref{nest2} in $\Hinf(\widetilde H_0-\widetilde C)$. (This is clear if $X-\int H_0$ is smoothable, and follows as in the last paragraph of the proof of Theorem~\ref{collar} otherwise.) One application is to distinguish infinitely many smooth structures on a much larger class of manifolds \cite{MinGen}:

\begin{thm}\label{endsystemdiff}
Let $X$ be a connected topological 4--manifold (possibly with boundary). Let $\pi\co\widetilde X\to X$ be the identity on $X$ if the latter is orientable, and its orientable double covering otherwise. Suppose that there is a proper topological embedding $H_0-C\to X$ as above, with $\pi$ extending over $H_0$ so that $\iota_\infty$ is injective, and with controlled instability relative to a collection $\{H_i^s\}$ with some $B_i\ne0$. If $X$ has just one end, assume that $ks(X-\int H_0)=0$. Then $G_\infty$ on $\widetilde X$ distinguishes infinitely many diffeomorphism types of smooth structures on $X$, and uncountably many if $\lim_{i\to\infty}\Rk B_i$ is infinite. The same holds for nonorientable $X$ with $\pi=\id_X$, provided that $H_0$ is orientable.
\end{thm}

\noindent The main issue in the proof is to control diffeomorphisms sending ends of $H_0-C$ into its complement. This is in the spirit of the second paragraph of the proof of Theorem~\ref{collar}, but more delicate. Note that when $X$ has more than one end, we can assume (after replacing $H_0$ by one component of some $H_i$) that some end lies outside $H_0$ so that $ks(X-\int H_0)=0$ trivially.

 \begin{example}\label{mess2}Let $X$ be a 4--manifold with $ks(X)=0$. Then for any topological embedding in $X$ of a nest of 2--handlebodies $H_i$ as in Example~\ref{mess}, we obtain infinitely many diffeomorphism types of smoothings on $X-C$, although it need not have a collared end.
\end{example}

\section{Topological submanifolds}\label{Submfd} 

We now digress to illustrate how our theory can be applied to topological submanifolds of smooth 4--manifolds. We will see that minimal genera of neighborhoods of such submanifolds can often be flexibly adjusted by topological ambient isotopy. This leads to a notion of minimal genus for any subset of a smooth 4--manifold, distinguishing topologically ambiently isotopic embeddings of surfaces and 3--manifolds. The question immediately arises:\ distinguishing up to what notion of equivalence? We define an equivalence relation between subsets of smooth manifolds that is weak enough to overlook various local pathologies, but still conclude that the topological ambient isotopy class of a typical tame surface or 3--manifold contains many of these ``myopic equivalence'' classes. (Recall that a topologically embedded submanifold of $X$ is {\em tame} if it has a topological tubular neighborhood, ie if it is the 0--section of a vector bundle topologically embedded in $X$ with codimension 0.) 

We begin by examining topologically embedded 3--manifolds, then proceed to surfaces.

\begin{thm}\label{M3iso} Let $M^3$ be the boundary of a compact, connected, orientable 2--handlebody $H$ topologically embedded in a smooth 4--manifold $X$. Assume $M$ is tame in $X$. For the inclusion--induced map $\iota\co {\rm H}_2(M)\to {\rm H}_2(H)$, let $D_1\subset\cdots\subset D_n=\im \iota$ be a filtration by direct summands. Then there is an arbitrarily rapidly increasing sequence $\{k_i\thinspace |\thinspace 1\le i<n\}$ and a topological ambient isotopy of $X$ after which $M$ has arbitrarily small topological tubular neighborhoods $U$ in which each $\iota^{-1}D_i$ is the span of all smooth surfaces in $U$ with genus at most $ k_i$.
\end{thm}

\noindent Such embeddings of 3--manifolds often arise in practice. Many constructions of closed, smooth and symplectic 4--manifolds involve gluing pieces together along smooth 3--manifolds, and frequently one piece fits the above description of $H$ with $\iota$ nonzero. For example, the preimage of a generic disk under a Lefschetz fibration with closed fibers always has such description.

\begin{proof} For each $i$, let $A_i$ be the rational span of $D_i$ in ${\rm H}_2(H)$, so that $A_1\subset\cdots\subset A_n\subset{\rm H}_2(H)$ is a filtration by direct summands, and $D_i=A_i\cap\im\iota$ for each $i$. Since ${\rm H}_2(H)$ is a direct summand in the 2--chain group of $H$, we can slide handles so that each $A_i$ is ${\rm H}_2(H_i^s)$ for some subhandlebody $H_i^s\subset H$ containing all 1--handles of $H$. We interpret these subhandlebodies as in Lemma~\ref{nest}. Since $H_1^s$ contains all of the 1--handles, the Mayer--Vietoris sequence for $H_i^s\cup M$ shows (as at the end of the proof of Lemma~\ref{nest}) that $B_i$, which by definition is the image of ${\rm H}_2(H_i^s\cap M)\to{\rm H}_2(H)$, equals $A_i\cap\im\iota=D_i$. The same sequence with $i=1$ shows that $\ker\iota$ is the image in ${\rm H}_2(M)$ of $\ker(j\co{\rm H}_2(H_1^s\cap M)\to{\rm H}_2(H_1^s))$. Apply Lemma~\ref{nest2} (or \ref{nest} if $\iota$ is injective) with $\Sigma$ induced by inclusion. Using Remark~\ref{nestadd}, we can represent a finite spanning set for $\ker\iota=\im(\ker j)$ by $\varphi_1^*\Sigma$--smooth, proper embeddings $F\times[0,\infty)\to H_1-C$ with $k_1$ exceeding their genera, and represent a similar set for each $D_i$ by such embeddings with $g(F)\le k_i$. Then each $\iota^{-1}D_i$ is spanned by surfaces (smooth in $X$) with genus at most $ k_i$, lying in $\varphi_1(\partial H_{n+1})$. Since the latter is ambiently isotopic to $M$ via the canonical product structure, the theorem follows for any $U$ enclosed by $\varphi_1(\partial H_n)$.
\end{proof}

\begin{thm}\label{Fiso}
In a smooth 4--manifold $X$, let $F$ be a compact, connected, tame surface (without boundary). Then there is an arbitrarily large integer $m$ and a topological ambient isotopy after which $F$ has arbitrarily small topological tubular neighborhoods for which $m$ is the minimal genus of the generator in the minimal cover of the neighborhoods for which the surface and 4--manifold are orientable. Alternatively, the isotopy can be chosen so that there is a tubular neighborhood system for which that minimal genus increases without bound. When $X$ is orientable and $F$ is originally smoothly embedded, one can realize any $m\ge g(F)$ if $F$ is orientable, and otherwise any $m\ge g(\widetilde F)$ with $m\equiv g(\widetilde F)$ mod 2 in the double cover. In each case, the resulting tame surface can be assumed almost smooth, ie smooth except at one point.
\end{thm}

\begin{proof} Freedman's original paper \cite{F} shows that embedded core disks of Casson handles can be taken to be almost smooth, but rather than introduce his machinery, we import the key ideas into our present setup. For any smooth manifold $W$ homeomorphic to $D^2\times\R^2$, the proof of Lemma~\ref{nest} showed how to isotope so that $D^2\times D^2$ maps to a topological 2--handle $h$ inside a smooth Casson handle $CH$ in $W$, with $h$ obtained from the first stage $T$ of $CH$ (squeezed into $h$) by adding topological 2--handles $h_r$. We need to further improve the picture by a topological isotopy in $CH$ that is smooth near $T$. First, perform the same construction in each $h_r$ to get a smaller topological 2--handle $h'_r\subset CH_r\subset h_r$, with $h_r'$ topologically ambiently isotopic to $h_r$ rel its attaching circle. By compactness, $h'_r$ lies in some finite subtower $T_r$ of $CH_r$. Let $T^*\subset\int CH$ be obtained by shrinking $T$ away from the boundary of $W$, so that the attaching circles of $CH$ and $T^*$ are connected by an annulus in $CH-\int T^*$, and then attaching the towers $T_r$. Then $T^*$ is a Casson tower, so it is diffeomorphic to a tubular neighborhood of a wedge of circles. (The core of the tower can be collapsed from the bottom up.) Since $CH$ is simply connected, $T^*$ can be assumed after smooth isotopy to lie in a preassigned open subset $V$ of $CH$. Then $T$ is seen as an arbitrarily small, smooth, 1--stage tower in $\int CH$, connected to the attaching circle of $CH$ along a smooth annulus $A$ in $CH$. The 2--handles $h'_r$ also lie in $V$, disjointly from $A$, and attaching them to $T$ gives a topological 2--handle $h'\subset V$ which, when extended by a tubular neighborhood of $A$, becomes topologically ambiently isotopic to $h$, smoothly near $T$. In fact, up to a compactly supported topological isotopy of $\int W=\R^2\times \R^2$, $h'$ is a standardly embedded $D^2\times D^2$, connected to $\partial W$ by the standard unknotted annulus $A$.

Now for $F\subset X$ as given, let $H_0\subset X$ be a $D^2$--bundle neighborhood of $F$, realized as a handlebody with a single 2--handle. By the method of Lemma~\ref{nest}, we construct a canonical nest of topological 2--handlebodies in $H_0$. (Smooth the 0-- and 1--handles, then construct nested embeddings of the 2--handle inside  Casson handles.) In the minimal cover to which $F$ and $X$ lift orientably, we can either force the minimal genus of the generator to increase without bound (by omitting the 2--handle from each $H_i^s$), or preserve it as some arbitrarily large integer (by setting $H_2^s=H_2$). If $X$ is orientable and $F$ is smooth, then the proof of Theorem~\ref{bundle} gives the initial Casson-type smoothing $U_1$ for realizing a preassigned $m$ as specified above. In each case, we assume each 2--handle is determined by $h'\cup A$ as in the previous paragraph, inside the corresponding handle $h'$ of the previous stage, with the nested handles $h'$ intersecting in a single point $x$. Then the common intersection $C$ of the handlebodies is an almost-smooth surface topologically ambiently isotopic to $F$: The core disk of the original 2--handle $h$ of $H_0$ has been removed from $F$ and replaced by the infinite union of the consecutive smooth annuli, one-point compactified at $x$. Since each annulus is standard up to isotopy in the previous $h'$, we can consecutively isotope the handles $h'$ and annuli onto a standard model in $H_0$. The limiting isotopy extends continuously over $x$, isotoping $C$ onto $F$ as required.
\end{proof}

Clearly, these theorems smoothly distinguish infinitely many embedded submanifolds within a given topological ambient isotopy class, up to some equivalence relation. It is uninteresting that the germs of the embeddings are distinct, since germs are sensitive to tiny changes such as smoothing a corner. Instead, we define a much weaker equivalence relation, by generalizing \cite{kink}:

\begin{de} Two subsets $Z_i\subset X_i$ of smooth manifolds are {\em myopically related} if for any neighborhoods $U_i$ of $Z_i$ ($i=1,2$) there are smaller neighborhoods $V_i$ that are diffeomorphic to each other, with each inclusion $Z_i\subset V_i$ a homotopy equivalence. (Thus, we can't distinguish the subsets without glasses). Two subsets are {\em myopically equivalent} if they are connected by a finite sequence $Z_j\subset X_j$ of subsets for which consecutive pairs are myopically related.
 \end{de}

 \noindent We may also restrict how the diffeomorphisms behave on homology or orientations. However, we don't require the diffeomorphisms to preserve the subsets $Z_i$, which would imply that the germs are the same. We then lose the obvious proof of transitivity, but obtain (for those $Z_i$ having suitable neighborhoods for reflexivity, such as tame submanifolds) a much weaker equivalence relation: Any two continuous sections of a given smooth vector bundle are myopically related, and some PL locally knotted surfaces in 4--manifolds are myopically related to smooth surfaces \cite{MinGen}.

 \begin{de}\label{mgsubset} For any subset $Z\subset X^4$ and $\alpha\in {\rm H}_2(Z)$, the {\em minimal genus of $\alpha$ in $Z$} is the supremum in $\Z^{\ge0}\cup\{\infty\}$ of the minimal genera of the images of $\alpha$ in  neighborhoods of $Z$.
 \end{de}

\noindent This agrees with previous usage when $Z$ is itself a 4--manifold (allowing a smooth boundary), and is an invariant of myopic equivalences preserving $\alpha$. We use this invariant, in the appropriate cover if necessary, to reinterpret the previous two theorems:

 \begin{cor}\label{infiso} In a smooth 4--manifold $X$, every tame compact surface $F$, and every 3--manifold $M$ as in Theorem~\ref{M3iso} with $\iota\ne0$, is ambiently isotopic to other submanifolds representing infinitely many myopic equivalence classes. One can realize arbitrarily large integers and $\infty$ as the minimal genus of the generator of $F$ (or its suitable cover). If $X$ is orientable and $F$ is smooth, one can realize any integer larger than $g(\widetilde F)$, provided that its parity agrees with that of $g(\widetilde F)$ if $F$ is nonorientable. For $M$, if $\im\iota$ has rank at least $2$, there are (infinitely many) pairs of submanifolds ambiently isotopic to $M$ such that each has a topological tubular neighborhood in which no tubular neighborhood of the other embeds smoothly and injectively on ${\rm H}_2$.
 \end{cor}

 An analog of the last sentence for surfaces is given in Corollary~\ref{infisoF}.

 \begin{proof} The assertions about $F$ are immediate. For $M$, note that a smooth embedding $U^4\to V^4$ that is injective on ${\rm H}_2$ gives a pointwise inequality $\gamma_U\le\gamma_V$ of genus--rank functions. Thus, a tame, compact submanifold inherits a limiting genus--rank function $\bar\gamma\co \Z^{\ge0}\to\Z^{\ge0}$ from a tubular neighborhood system (independent of the choice of system). This is an invariant of myopic equivalence, and distinguishes infinitely many classes of submanifolds ambiently isotopic to $M$, obtained by setting $n=3$ and $0=D_1\ne D_2=\im\iota$ in Theorem~\ref{M3iso}. It now suffices to prove the last sentence of the corollary, by finding two submanifolds $M_i$ ambiently isotopic to $M$, with functions $\bar\gamma$ not related by an inequality. We construct these simultaneously. Choose a nontrivial, proper direct summand $D$ of $\im\iota$. Start constructing $M_1$ to obtain $k$ for which $\iota^{-1}D$ will be the span of all surfaces of genus at most $ k$ (with an arbitrarily large lower bound on the genera of surfaces not in $\ker\iota$). Construct $M_2$ with a neighborhood in which all surfaces not in $\ker\iota$ have genus larger than $k$, and let $k'$ exceed the genera of a spanning set for the homology. Then complete the construction of  $M_1$ so that all surfaces outside $\iota^{-1}D$ have genus larger than $k'$.
 \end{proof}

In contrast to surfaces, the 3--manifolds  constructed in Theorem~\ref{M3iso} with $\iota\ne0$ typically have neighborhoods $U$ in which they cannot be made smooth away from a point by isotopy. In fact, the nonsmooth set must intersect every embedded surface $F\subset M$ whose genus is strictly smaller than $G([F])$ in $U$. This is consistent with Quinn's work showing that smoothing in dimension 4 breaks down on a set with codimension~2 intersection properties. For an example without homology, a different method shows that Freedman's tame embedding of the Poincar\'e homology sphere $\Sigma$ in $S^4$ cannot be isotoped so that smoothness fails only on some 3--ball embedded in $\Sigma$ with tame boundary \cite{MinGen}.

\section{Uncountable families with the same genus functions}\label{Uncountable}

It should be apparent from previous sections that the exotic behavior detected by the genus functions is quite different from that detected by older invariants that were developed to study smoothings of the contractible manifold $\R^4$. To see this contrast more clearly, we now combine the two technologies. We show that under various hypotheses, our previous results hold for uncountably many diffeomorphism types for each choice of genus data. (Recent work of Bennett \cite{B} shows that the Taylor invariant also can often be controlled independently.) Our techniques are compatible with the Stein condition, allowing us to prove a corollary for classical complex analysis (Corollary~\ref{holomorphy}) that yields uncountable families of exotic Stein domains as stated in Theorem~\ref{introHolomorphy}. This follows from the more general Theorem~\ref{uncountable} involving manifolds with definite intersection forms. (For an even more general version, see Remark~\ref{minuscompact}.)

\begin{thm}\label{uncountable} Let $\pi\co\widetilde X\to X$ be a covering of noncompact 4--manifolds, with $X$ connected and $\widetilde X$ oriented, lifting an orientation on $X$ if one exists. Suppose there is a smoothing $\Sigma$ on $X$ such that each compact, codimension--0 submanifold of $\widetilde X$ can be $\pi^*\Sigma$--smoothly embedded (preserving orientation) into a smooth, closed, simply connected, negative definite 4--manifold. Then there are other such smoothings $\Sigma_t$ of $X$, whose lifts $\pi^* \Sigma_t$ realize uncountably many diffeomorphism types on each component of $\widetilde X$. Each $\Sigma_t$ embeds in $\Sigma$ and is stably isotopic to it, has the same genus functions $G$ and $G_\infty$ as $\Sigma$, and for each $\alpha\in {\rm H}_3(X)$ has the same 3--manifolds arising as smoothly embedded representatives of $\alpha$. The smoothings  $\pi^*\Sigma_t$ are similarly related to  $\pi^*\Sigma$. If $\Sigma$ (resp. $\pi^*\Sigma$) has a Stein structure respecting the orientation on $\widetilde X$, then so does each $\Sigma_t$ (resp. $\pi^*\Sigma_t$ if $X$ is orientable).
\end{thm}

\noindent It follows that the smoothings $\Sigma_t$ realize uncountably many diffeomorphism types whenever there is a base point for which $\im \pi_1(\widetilde X)\subset \pi_1(X)$ has only countably many images (up to inner automorphism) under the homeomorphism group of $X$, as is automatically true if $\pi_1(X)$ is finitely generated. Note that there is no claim about 3--manifolds at infinity in the spirit of Remark~\ref{H3}(b). The proof shows that any sequence of 3--manifolds approaching the end of some $\Sigma_t$ can also be found in $\Sigma$ (and similarly for $\widetilde X$), but the converse fails: If $\Sigma$ is the standard smooth structure on $\R^4$, then every compact subset is surrounded by a smooth 3--sphere, whereas this fails for the smoothings $\Sigma_t$ arising in the proof (and for all exotic $\R^4$'s if the 4--dimensional smooth Poincar\'e Conjecture holds).

All of our previously exhibited genus functions can now be realized by uncountably many diffeomorphism types whenever the definiteness condition is satisfied \cite{MinGen}. Here are some highlights:

\begin{cor}[a]\label{combined}
Suppose an orientable $X$ is an $\R^2$--bundle over a closed, orientable surface $F$. Then for any $g'\ge g(F)$ there are uncountably many smoothings for which the minimal genus of a generator of ${\rm H}_2(X)$ is $g'$. If $F$ is nonorientable, the same holds for smoothings of $X$ with minimal genus $g'-1$ in the canonical double cover of $X$, provided that $g'\equiv g(F)$ mod 2. These smoothings admit Stein structures oriented with $e(X)\le 0$, with the exceptions given in Theorem~\ref{bundle}.

\item[b)] For an end-sum of $\R^2$--bundles $X_i$ over surfaces, one can distinguish infinitely many smoothings (uncountably many for an infinite end-sum) by the genus function as in Example~\ref{bundles}(b), and each resulting genus function is realized by uncountably many diffeomorphism types provided  that all $X_i$ are orientable with Euler numbers of the same sign (allowing 0), or that each orientable $X_i$ has $e=0$.

\item[c)] For $S^1\times\R^3$, every $k\in\Z^{\ge0}\cup\{\infty\}$ is realized as the minimal genus of the generator at infinity for uncountably many smoothings.

\item[d)] If a 3--manifold $M$ topologically embeds in $\# n\overline{\C P^2}$ as in Theorem~\ref{M3iso}, then the preimage under $\iota$ of any filtration of $\im\iota$ by direct summands lies in the genus filtration of some smoothing of $M\times\R$, with corresponding characteristic genera arbitrarily large, and the same genus function is realized by uncountably many diffeomorphism types embedding in $\# n\overline{\C P^2}$. One can also similarly control minimal genera of one end with a second filtration of $\im\iota$, in the manner of Theorem~\ref{MxR}.
\end{cor}

\begin{proof} It suffices to check the definiteness condition in the appropriate cover in each case; then Theorem~\ref{uncountable} applies to the relevant theorems and examples. For (a) and (b), note that every orientable $\R^2$--bundle over an orientable surface smoothly embeds in $\# n\overline{\C P^2}$ for some finite $n$ (eg by blowing up the trivial bundle in $S^4$). Thus, the relevant Stein--Casson smoothings also embed. An infinite end-sum of such bundles need not have such an embedding, but every compact subset of it does if the Euler numbers have the same sign. If the total space of a bundle is nonorientable, then its orientable double cover has Euler number 0. For (c), the smoothings from Theorem~\ref{1h} embed in $S^4$ by construction (also yielding a result about arbitrary 1--handlebodies). For (d), combine the proofs of Theorems~\ref{M3iso} and \ref{MxR}.
\end{proof}

\begin{proof}[Proof of Theorem~\ref{uncountable}] If $R$ is an exotic $\R^4$, we can create a new smooth structure $\Sigma'$ on $X$ by forming an {\em end-sum} of $\Sigma$ with $R$ (cf.\ \cite{infR4}, \cite{GS}): Choose rays $\gamma$ and $\gamma_R$ smoothly embedding $[0,\infty)$ into $X_\Sigma$ and $R$, respectively, with tubular neighborhoods $U$ and $U_R$ bounded by smoothly, properly embedded copies of $\R^3$. Choose a homeomorphism (preserving orientations if they have been specified) from $R-U_R$ to $\cl U$ that is smooth near the boundary. (To find such a homeomorphism, use Quinn's handle straightening to smooth $\gamma_R$ and $U_R$ in the standard smoothing of $\R^4$, and conclude that they are topologically standard.) Then push forward the smooth structure of $R$ to $U$ and extend it to agree with $\Sigma$ outside $U$. The resulting smooth manifold $X_{\Sigma'}$ can be thought of as $X_\Sigma$ glued together with $R$ at infinity. Similarly, we can sum with infinitely many exotic $\R^4$'s using a proper embedding of $\Z_+\times [0,\infty)$. The stable isotopy class of $\Sigma$ is preserved under such end-sums since $\cl U\approx \R^3\times [0,\infty)$ has a unique stable isotopy class of smoothings.

\begin{remark} The isotopy and diffeomorphism class of an end-sum of $\Sigma$ with $R$ depend on the end of $X$ specified by $\gamma$, and the local orientations, but are otherwise well defined if, for example, the end is simply connected or topologically collared, or $X$ is a handlebody interior with only finitely many 3--handles \cite{CH}.
\end{remark}

 If $R$ embeds smoothly in $S^4$, then $\Sigma'$ and its lift satisfy all of our requirements for a smoothing $\Sigma_t$ except possibly the Stein condition. The first step in proving this is to construct smooth embeddings $i\co X_\Sigma\to X_{\Sigma'}$ and $j\co X_{\Sigma'}\to X_\Sigma$ that have topological isotopies to the identity $\id_X$. For $i$, begin with the $\Sigma$--smooth family of embeddings $\gamma_s(t)=\gamma(t+s)$, for $0\le s<\infty$, sliding the ray $\gamma$ out to infinity. This can be extended to a smooth family of diffeomorphisms of $X$, whose limit as $s\to\infty$ has domain $X-\cl U$ and is the inverse of $i$. For $j$, choose smooth balls $B\subset U$ and $B_R\subset U_R$, then embed $R-\cl U_R\subset S^4$ into $X_\Sigma$  by identifying $S^4-\int B_R$ with $B\subset U\subset X_\Sigma$. Connecting the boundaries of $U$ and $U_R$ by a thickened smooth arc gives the embedding $j\co X_{\Sigma'}\to X_\Sigma$ as the end-sum of $X_\Sigma-\cl U$ with $R-\cl U_R\subset B$ along the arc. This $j$ is topologically isotopic to $\id_X$ by first radially shrinking the half-space $R-U_R$, then absorbing the arc into $X-\cl U$, and finally applying our previous isotopy of $i$.

 To verify the conclusions of Theorem~\ref{uncountable} for such a smoothing $\Sigma'$, note that $i$ sends any smooth submanifold of $X_\Sigma$ to one in $X_{\Sigma'}$ representing the same homology class, and $j$ does the same in the opposite direction. The statements about $G$ and ${\rm H}_3$ follow immediately for both $\Sigma'$ and $\pi^*\Sigma'$ (since the latter is obtained from $\pi^*\Sigma$ by a multiple end-sum), and  $\pi^*\Sigma'$ inherits the given relation to definite manifolds. To show that end-summing with $R$ cannot increase values of $G_\infty$, note that any smooth surface in $X_\Sigma$ can be pushed off of the 1--manifold $\im\gamma$ while avoiding a preassigned compact subset. (This fails for embedded 3--manifolds.) For the reverse inequality, it suffices to show that for every neighborhood $X-K$ of infinity in $X_\Sigma$, there is one in $X_{\Sigma'}$ with a smooth (not proper) embedding in $(X-K)_\Sigma$ that is isotopic in $X-K$ to the inclusion. But this is easily arranged by repeating our construction of $j$ after truncating $\gamma$ to be disjoint from $K$. Since end-summing is equivalent to attaching the corresponding handlebodies by 1--handles, $X_{\Sigma'}$ will be Stein if both  $X_\Sigma$ and $R$ are (since we assume the orientations are respected), and similarly for $\widetilde X$. We obtain the necessary Stein exotic $\R^4$'s below.

 To obtain an uncountable family, we must choose $R$ with additional care. The simplest known exotic $\R^4$, which we will denote $R^*$, was constructed in \cite{BG}. (See also  \cite[Theorem~9.3.8]{GS}.) This smoothly embeds in $S^4$, and has Stein--Casson type \cite{Ann}, built with two 1--handles, a 2--handle and the Casson handle with a positive, unique double point at each stage. It arose as the limit of a sequence of exotic $\R^4$'s associated to a sequence of nontrivial h-cobordisms of elliptic surfaces. The $\R^4$'s in the sequence are each obtained from $R^*$ by refining the Casson handle (above successively higher stages), and can be assumed (when suitably refined) to retain Stein--Casson type. Let $R_1$ be the first element of this sequence, associated to the h-cobordism $W$ from $\partial_-W=K3\#\overline{CP^2}$ to $\partial_+W$, a sum of $\pm \C P^2$'s. (Presumably $R^*$ could be used in place of $R_1$ below, but the required gauge theory doesn't seem to have been worked out carefully beyond the first h-cobordism in the sequence.) The h-cobordism $W$ is a product outside a compact subset intersecting $\partial_-W$ in a compact, codimension--0 submanifold $K$ contained in, and oriented by, an embedded copy of $R_1\subset\partial_-W$, and $\partial_+W$ is obtained from $\partial_-W$ by removing a large neighborhood of $K\subset R_1$ from $\partial_-W$ and regluing it by a diffeomorphism near the end of $R_1$. We can now obtain a family $\{R_t|t\in Q\}$ of Stein exotic $\R^4$'s, indexed by the standard Cantor set $Q\subset I=[0,1]$, such that $K\subset R_s\subset R_t$ with compact closure whenever $s<t$: First note that any given compact subset $L\subset R_1$ has intersection with the Casson handle lying in some finite subtower $T_n$. The complement of $T_n$ in the Casson handle is a disjoint union of higher-stage Casson handles. Setting $L=K$ and applying Quinn's Handle Straightening Theorem (see proof of Lemma~\ref{nest}), we can find Casson handles inside these, which when attached to $T_n$ produce $R_{1/3}$, a new Stein exotic $\R^4$ such that $K\subset R_{1/3}\subset R_1$ with compact closure. Iterating, for all endpoints in $Q$ of middle thirds, and passing to the limit as in \cite{F}, gives the required family. (Such families of nested Casson handles indexed by $Q$ were a crucial part of Freedman's proof that Casson handles are homeomorphic to 2--handles. However, they are much harder to construct without presupposing Freedman's Theorem through Quinn's.) A gauge-theoretic argument developed by DeMichelis and Freedman \cite{DF} shows that each diffeomorphism type in the family $\{R_t\}$ occurs only countably often, whereas $Q$ has the cardinality of the continuum, so there are uncountably many diffeomorphism types (with this same cardinality in ZFC set theory). If the smooth structure $\Sigma_t$ on $X$ is obtained by end-summing $\Sigma$ with $R_t$ as above, then Lemma~\ref{radial} below, with each $X_t$ a component of $\widetilde X_{\pi^*\Sigma}$ independent of $t$, generalizes this result to the family $\{\pi^*\Sigma_t\}$, completing the proof.
\end{proof}

\begin{lem}\label{radial} For $(R_1,K)$ as in the previous proof and a subset $Q\subset I$ containing 1, let $\{(R_t,K)|t\in Q\}$ be a family of smooth manifolds homeomorphic to $\R^4$, each containing a smoothly embedded copy of $K$, such that whenever $s,t\in Q$ with $s<t$, there is a smooth embedding $R_s\emb R_t$ whose image has compact closure, restricting to the identity on $K$. Let $\{X_t|t\in Q\}$ be any family of smooth, oriented, noncompact 4--manifolds (allowing boundary) such that each compact, codimension--0 submanifold of each $X_t$ smoothly embeds, preserving orientation, into a smooth, closed, simply connected, negative definite 4--manifold. For each $t\in Q$, let $X_t'$ be obtained from $X_t$ by end-summing with (possibly infinitely many) copies of $R_t$, with at least one sum respecting the orientations of $X_t$ and $K$. Then in the family  $\{X_t'|t\in Q\}$, each diffeomorphism type appears only countably often.
\end{lem}

\noindent While we did not need $X_t$ to depend on $t$, this generality causes no difficulties and will be needed in \cite{steintop} to give a sharper version of our Corollary~\ref{holomorphy}.

\begin{proof} We use the method of \cite{DF} augmented as in \cite{menag}. First we show that for $s,t\in Q$ with $s<t$, no diffeomorphism (or even embedding) can map $X_t'$ into $X_s'$ so that the copy of $K$ in one $R_t$ summand maps by the identity to $K$ in some $R_s$ summand respecting the orientation of $X_s$. Otherwise, we could find an embedding $i\co R_s\emb R_s\# N$, for some negative definite $N$ as given in the lemma (oriented compatibly with $K\subset R_s$), with $i|K=\id_K$, and with $i(R_s)$ having compact closure. We will see below that this is a contradiction. To construct $i$, first compose the given rel $K$ embeddings, $R_s\emb R_t\emb X_t'\emb X_s'$ (so compact closure is inherited from the first embedding). Then embed the latter into the end-sum $X_s''$ of $X_s$ with just the given copy of $R_s$, by eliminating any other summands as in the previous proof, using the fact that $R_1$ embeds in $R^*$ and hence in $S^4$. The image of $R_s$ under the composite embedding lies in a compact subset of $X_s''$, which we can assume (after enlarging as necessary) is a 4--manifold intersecting the separating $\R^3$ in a 3--ball. Thus, it can be written as the boundary sum $K_1\natural K_2$ of two pieces with $K\subset K_1\subset R_s$ and $K_2\subset X_s$. By hypothesis, $K_2$ then embeds in a closed, negative definite $N$. Form the connected sum $R_s\# N$ using balls disjoint from $K_1,K_2$ but near their correct boundary components so that we may ambiently add a 1--handle to get the required embedding $R_s\emb K_1\natural K_2\emb R_s\# N$.

The required contradiction now arises exactly as in the proof of \cite[Lemma~1.2]{menag}. Briefly, for the cobordism $W$ in the previous proof, we have embeddings $R_s\emb R_1\emb \partial_- W$ rel $K$ and $R_s-K\emb \partial_+ W$ that are identified outside of $K$ by the product structure on $W$. By repeatedly summing $\partial_-W$ with $N$ and applying the embedding $i$, we obtain $\partial_-W\#n N$ with $n$ consecutive rings, each diffeomorphic to $R_s\# N$ with a fixed compact subset of $i(R_s)$ removed, separating a neighborhood of $K$ from the rest of the manifold. Removing the neighborhood of $K$ and letting $n\to\infty$, we obtain an open 4--manifold with a negative definite, periodic end. A similar construction applies to $\partial_+ W$, and these two end-periodic manifolds are diffeomorphic by construction. However, the manifolds $\partial_\pm W$ are distinguished by a gauge-theoretic invariant, which by work of Taubes \cite{Taubes} (followed by \cite{DF}) is well-defined on the end-periodic manifolds but unchanged during their construction from $\partial_\pm W$. Since the end-periodic manifolds are diffeomorphic, this is a contradiction.

Now for $s\ne t$, suppose there is a diffeomorphism $\varphi\co X_t'\to X_s'$. For any choice of end-summands respecting orientations, we have shown that $\varphi$ cannot be the identity on the corresponding copies of $K$. Given another diffeomorphism $\psi\co X_r'\to X_s'$ with $r\ne t$, then $\varphi|K$ and $\psi|K$ cannot be smoothly isotopic for any such $K$, or else $\psi^{-1}\circ\varphi$ would contradict this assertion. Thus, the set of $t$ for which $X_t'$ is diffeomorphic to a given $X_s'$ has cardinality bounded by that of the set of isotopy classes of embeddings $K\emb X_s'$. However, this latter cardinality is countable since $K$ is compact. (For example, pass to the PL category and note that any fixed triangulations of $K$ and $X_s'$ support only countably many embeddings, then note that any given embedding is captured after finitely many barycentric subdivisions.)
\end{proof}

\begin{Remark}\label{minuscompact} For $\widetilde X$ connected, Theorem~\ref{uncountable} remains true under the weaker hypothesis that $\widetilde X$ has a closed, noncompact subset $Y$ with a compact, $\pi^*\Sigma$--smooth 3--manifold boundary, for which every compact submanifold of $Y$ embeds as specified in the theorem \cite{MinGen}. This is because $Y$ has uncountably many smoothings but $\partial Y$ has only countably many embeddings in a given smoothing of $\tilde X$. The theorem also sharpens Lemma~\ref{nest2}: If $\int\widetilde H_0$ is connected and embeds $\pi^*\Sigma$--smoothly in a smooth, closed, simply connected, negative definite 4--manifold, then each combination of $G$ and $G_\infty$ that is realized by some $\varphi_1$ on $\int H_0-C$ and $\int\widetilde H_0-\widetilde C$ can also be realized by other isotopies yielding uncountably many smoothings on $\int\widetilde H_0-\widetilde C$. Similar families arise for coverings $\pi\co\widetilde X\to X$ whose smoothings are obtained by applying Lemma~\ref{nest2} to a proper topological embedding of $H_0-C$. (Locate $Y$ as above in $\int\widetilde H_0-\widetilde C$.) This gives some improvements in Section~\ref{Apps}.
\end{Remark}

It is a classical problem to understand which open subsets of $\C^2$ are {\em domains of holomorphy}, that is, are Stein surfaces in the complex structure inherited from the embedding. The question is studied in  $\C^2$ and in more general complex surfaces $S$ in \cite{JSG}, \cite{steindiff}, where the subset is allowed to vary by topological, resp.\ smooth, isotopy of its inclusion map. Theorem~\ref{uncountable} augments these results to show that such Stein open subsets frequently occur in isotopic families realizing uncountably many diffeomorphism types.

\begin{cor}\label{holomorphy} Let $S$ be a complex surface  that smoothly embeds (preserving orientation) in a possibly infinite blowup of  $\C^2$. Let $f\co X\emb S$ be a topological embedding of an open 4--manifold.
\begin{itemize}
\item[a)] Suppose $f(X)$ is a Stein open subset of $S$. Then $f$ is topologically isotopic in $S$ to other embeddings with  uncountably many diffeomorphism types of Stein images, whose genus functions $G$ and $G_\infty$ are the same as for $f$.
\item[b)] Alternatively, suppose $X$ is homeomorphic to the interior of a 2--handlebody. Then $f$ is topologically isotopic to embeddings with uncountably many diffeomorphism types of Stein images, whose genus functions $G$ and $G_\infty$ all agree, and for which $G$ can be controlled as in Lemma~\ref{main}. If ${\rm H}_2(X)$ is finitely generated, then any maximal filtration of it can be chosen to be the genus filtration as in Theorem~\ref{filt}.
    \end{itemize}
\end{cor}

\begin{Remarks} (a) While we do not conclude that the isotopies are ambient, this can be arranged if $f$ extends smoothly (resp.\ tamely and properly) over the boundary of the relevant handlebody \cite{steintop}.

(b) The hypothesis embedding $S$ includes the classical case $S=\C^2$, but is still restrictive. However, many indefinite complex surfaces, such as those on the BMY line and bundles over Riemann surfaces, are covered by open subsets of $\C^2$. The corollary easily extends to include such cases \cite{MinGen}.
\end{Remarks}

\begin{proof} To prove (a) of the theorem, note that while $S$ may be an infinite connected sum with copies of $\overline {\CP^2}$, every compact subset of $f(X)$ only intersects finitely many summands, and so lies in a finite blowup of $S^4=\C^2\cup\{\infty\}$. Thus, Theorem~\ref{uncountable} applies, with the smoothing $\Sigma$ induced by $f$, yielding exotic smoothings $\Sigma_t$ obtained by end-summing $\Sigma$ with $R_t$. These all admit Stein structures respecting the complex orientation of $X$, since $\Sigma$ does by hypothesis. Each $X_{\Sigma_t}$ smoothly embeds into $X_\Sigma$, and hence into $S$, and (by the proof of Theorem~\ref{uncountable}) the latter embedding $f_t$ is topologically isotopic to $f$ rel the core 2--complex $C$ of $X$ (which is the interior of a 2--handlebody). By \cite{steindiff} (cf.\ also \cite[Theorem~13.8]{CE}), an embedding onto an open subset of a complex surface is smoothly isotopic to one with Stein image if and only if the induced complex structure on the domain is homotopic (through almost-complex structures) to a Stein structure on it. This condition is trivially satisfied for $\Sigma$. Since the Stein structures of $\Sigma$ and $\Sigma_t$ are equal near $C$, which is a deformation retract of $X$, the condition is also satisfied by $\Sigma_t$, so $f_t$ is $\Sigma_t$--smoothly isotopic to an embedding in $S$ with Stein image.

Part (b) follows from \cite{JSG}, which proves that any embedding of such an $X$ into a complex surface is topologically isotopic to one with Stein image. Another way to view that theorem is that we can arrange $X$ to have Stein--Casson type, realizing the homotopy class of almost-complex structures induced by the embedding, then invoke \cite{steindiff}. To control $G$, further refine the Stein--Casson smoothing using Lemma~\ref{main} and Addendum~\ref{mainadd} (before applying \cite{steindiff}). The result now follows from (a).
\end{proof}

The hypothesis that subsets of $\widetilde X$ embed in negative definite manifolds can be bypassed by using a different technique for constructing uncountable families, related to the large exotic $\R^4$ construction:

\begin{thm}\label{large}
Let $X=X^*-\{x\}$, where $X^*$ is $\pm\C P^2$, $S^2\times S^2$, a K3--surface, or a finite connected sum of these. Let $\{A_i|0\le i\le b_2(X)\}$ be a filtration of ${\rm H}_2(X)$ by distinct direct summands. Then there are uncountably many diffeomorphism types of Stein smoothings on $X$, having genus filtration $\{A_i\}$ and all having the same genus--rank function, whose characteristic genera can be chosen to increase arbitrarily rapidly. The smoothings can be assumed to have the same values of $G(\alpha)$ for any preassigned finite collection of classes $\alpha\in {\rm H}_2(X)$.
\end{thm}

It is possible that every smooth, closed, simply connected 4--manifold is homeomorphic to some $X^*$ as given (with $S^4$ the trivial connected sum); this is the notorious $\frac{11}8$--Conjecture.

\begin{proof} Since the theorem is well-known when $X^*=S^4$, we may assume $b_2(X)\ne 0$, and orient $X$ so that it is not negative definite. Let $Y$ be Freedman's \cite{F}, \cite{FQ} closed, simply connected, topological 4--manifold with vanishing Kirby--Siebenmann invariant and intersection form $E_8\oplus\langle-1\rangle$ (negative definite but not diagonalizable). Then $Z=X^*\# Y$ also has vanishing Kirby--Siebenmann invariant but intersection form odd and indefinite, hence, diagonalizable. By Freedman's classification, $Z$ is homeomorphic to a connected sum of copies of $\pm\C P^2$, so it inherits a smooth structure. However $X^*$ was chosen, there is a finite 2--handlebody with interior $X$ (and no 1--handles). As in the proof of Theorem~\ref{filt}, we can assume each $A_i$ is the homology of a subhandlebody.  By Quinn's Handle Straightening Theorem (proof of Lemma~\ref{nest}), we can (nonambiently) isotope the embedding of $X$ in $Z$ so that it inherits a Casson-type smoothing. Applying Lemma~\ref{main} and Addendum~\ref{mainadd}, we can further refine the Casson handles so that the genus filtration is $\{A_i\}$ and $X$ has Stein--Casson type. We will construct uncountably many smoothings with the same genus filtration and genus--rank function as this one. Choose a finite collection of surfaces in $X$ including minimal genus representatives of each preassigned class $\alpha$ and a spanning set for each subgroup in the genus filtration as guaranteed by the definition of the latter. Let $T\subset\cl X\subset Z$ be obtained by replacing each Casson handle by its first few stages. By compactness, we can assume $T$ contains each surface in our finite collection, provided that enough stages of the Casson handles have been included in $T$. Then any refinement of $X$ containing $\int T$ must have the same genus filtration, genus--rank function, and values of each $G(\alpha)$ as before. As with $\{R_t\}$ in the last paragraph of the proof of Theorem~\ref{uncountable}, we can refine the Casson handles to construct a family $\{X_t\}$ of such Stein--Casson smoothings of $X$ indexed by the Cantor set, such that $X_s\subset X_t$ with compact closure whenever $s<t$. By a standard argument, eg \cite[Theorem~9.4.10]{GS}, no two of these are diffeomorphic: If $X_s$ and $X_t$ with $s<t$ even have diffeomorphic neighborhoods of their ends, let $W\subset X_t$ consist of the two diffeomorphic neighborhoods and the region lying in between. Then $W$ is homeomorphic to $S^3\times \R$, and we can smoothly glue copies of $W$ together near their ends to get other manifolds homeomorphic to $S^3\times \R$. Glue an infinite stack of copies of $W$ onto $Z-X_s$ to obtain a smooth manifold homeomorphic to $Y-\{y\}$ but with a periodic end. This contradicts a theorem of Taubes \cite{Taubes}.
\end{proof}

\section{Related Phenomena}\label{Related} 

In this final section, we extend our results in several directions. Having shown in previous sections how to force the genus functions $G$ and $G_\infty$ to be large, we now investigate how small they can be. We then amplify our invariants by allowing our surfaces to be immersed rather than embedded. This allows us to distinguish orientations and larger families of smoothings. We obtain deeper insight into topologically embedded surfaces and the diffeomorphism classification problem for Casson handles.

For obvious lower bounds on our genus functions, we define topological genus functions $G^{\top}$ and $G_\infty^{\top}$ as in the smooth case, but using tame, topologically embedded surfaces. These bound $G$ and $G_\infty$ below, for every smoothing on a given topological manifold. We now prove that these lower bounds are sharp for open 4--manifolds.

\begin{thm}\label{GTOP} For every open 4--manifold $X$, every stable isotopy class admits a smoothing for which $X$ and all of its covers satisfy $G=G^{\top}$ and $G_\infty=G_\infty^{\top}$.
\end{thm}

Sometimes such smoothings come in uncountable families, eg for $\R^4$. However, the construction uniquely picks an isotopy class of smoothings from each stable isotopy class, and the set of these isotopy classes is preserved under homeomorphisms. It seems unlikely that the theorem holds for closed 4--manifolds. There are tame topological surfaces in $\C P^2$ (for example) that violate the adjunction inequality \cite{R}, \cite{LW}. However, it seems hard to rule out exotic (nonsymplectic) smooth structures in which these surfaces could be smooth. The theorem yields a smoothing on $\C P^2-\{x\}$ with smaller genus function than the standard smoothing.

\begin{proof} Given a surface $F$ (possibly disconnected or noncompact) properly and tamely embedded in $X$, we first construct a smoothing $\Sigma_F$ of $X$ in which $F$ is smooth. Start with the standard smoothing on a normal disk bundle $N$ of $F$. Since $X-\int N$ has no compact components, we can smooth it rel $\partial N$ \cite{Q}. This smoothing fits together with the one on $N$ to give $\Sigma_F$. We can arrange $\Sigma_F$ to represent any stable isotopy class. This is because the map ${\rm H}^3(X-\int N,\partial N;\Z_2)\cong {\rm H}^3(X,F;\Z_2)\to {\rm H}^3(X;\Z_2)$ is surjective, so we can switch to any other stable isotopy class by changing the stable isotopy class on $X-\int N$, which can be chosen arbitrarily.

The required smoothing now results from a general construction of Freedman and Taylor \cite{FT} using $R_U$, their universal exotic $\R^4$. Given a smoothing $\Sigma$ of $X$, choose a smooth, proper embedding $\gamma\co [0,\infty)\times S\to X$ with discrete (hence countable) index set $S$, so that the component rays of $\gamma$ determine a dense subset of the space of ends $\E(X)$. Let $\Sigma^*$ be the result of end-summing $\Sigma$ with a copy of $R_U$ along each of these rays. As is proved in \cite{FT}, the isotopy class of $\Sigma^*$ is independent of choice of $\gamma$, and only depends on $\Sigma$ through its stable isotopy class. Thus, for a fixed stable isotopy class, the corresponding smoothing $\Sigma^*$ will have the desired properties: Every tame surface $F\subset X$ is smooth in $\Sigma_F^*$ (provided that we choose $\gamma$ and its tubular neighborhood disjoint from $F$), so the given isotopy between smoothings also smooths $F$ in $\Sigma^*$. Thus, $\Sigma^*$ satisfies $G=G^{\top}$. For $\alpha\in \Hinf(X)$, we can realize $G_\infty^{\top}(\alpha)$ by a tame, proper embedding of an infinite union of compact surfaces into $X$ (in successive neighborhoods of infinity), so we have $G_\infty=G_\infty^{\top}$. The same reasoning applies to the domain of any covering $\pi$ of $X$, since $\pi^*(\Sigma^*)$ is constructed from $\pi^*\Sigma$ using $R_U$ as above \cite{MinGen}.
\end{proof}

We can similarly deal with tame, compact 3--manifolds $M\subset X$ in place of embedded surfaces. The main differences are that a smoothing for which $M$ is smooth may not exist if $M$ cuts out a compact subset of $X$, and smoothness of $M$ restricts the stable isotopy class of the ambient smooth structure.

\begin{remark} While it is tempting to conjecture that every open 4--manifold has uncountably many smoothings (Question~\ref{Q}), some caution is suggested by the universal $\R^4$. Could there be a 4--manifold with ends constructed so that every smoothing must have the form $\Sigma^*$ as above? There would then be a unique isotopy class of smoothings in each stable class.
\end{remark}

Minimal genus techniques can be supplemented by considering generically immersed surfaces. For a smooth, oriented 4--manifold $X$ and $\alpha\in{\rm H}_2(X)$, let $g_{im}(\alpha)$ denote the minimal genus of an immersed (or equivalently, continuously mapped) surface representing $\alpha$. For example, this vanishes when $X$ is simply connected, and equals $g(F)$ when $\alpha$ is a generator for an $\R^2$--bundle over an orientable surface $F$.

\begin{de} The {\em kinkiness} of a class $\alpha$ is the pair $\kappa(\alpha)=(\kappa_+(\alpha),\kappa_-(\alpha))$, where $\kappa_+(\alpha)$ (resp.\ $\kappa_-(\alpha)$) is the minimum number of positive (resp.\ negative) double points of smoothly, generically immersed surfaces of genus $g_{im}(\alpha)$ representing $\alpha$.
\end{de}

\noindent This was first introduced in \cite{kink}, in the context of immersed disks in manifolds with boundary. Note that there may not be a single surface representing both $\kappa_+(\alpha)$ and $\kappa_-(\alpha)$ simultaneously. For example, if $X$ is obtained from $B^4$ by attaching a 2--handle along a $\pm 1$--framed Figure--8 knot $K$, and $\alpha$ generates its homology, then $\kappa(\alpha)=(0,0)$, although there is no embedded sphere because the boundary has nontrivial Rohlin invariant. (The relevant immersed spheres are constructed by noticing that $K$ can be unknotted by a single crossing change, and the sign of the crossing can be chosen arbitrarily since $K$ is amphichiral.) This also illustrates that there is no obvious general relation between kinkiness and the corresponding genus function, although if one counts double points without sign, one obtains an upper bound on both $G-g_{im}$ and $\kappa_++\kappa_-$. Kinkiness has the advantage of supplying two integer invariants, which can sometimes be controlled independently. For example:

\begin{thm}\label{kink} For an oriented $X$ that is an $\R^2$--bundle over an orientable surface $F$, let $\alpha\in {\rm H}_2(X)$ be a generator. Then for each pair $k,l\in\Z^{\ge 0}$, there are uncountably many diffeomorphism types of smoothings of $X$ with $\kappa(\alpha)=(k,l)$, provided that $|e(X)|\le 2g(F)$, or more generally, that $2k\ge\chi(F)+ e(X)$ (or $k=0$) and $2l\ge\chi(F)- e(X)$ (or $l=0$). Thus, every $X$ realizes all $k$ or all $l$ in this manner.
\end{thm}

\noindent This provides information that cannot be obtained from the minimal genus function. For example, those smoothings with $\kappa_+(\alpha)\ne\kappa_-(\alpha)$ admit no orientation-reversing self-diffeomorphism, even if $e(X)=0$. More complicated 2--handlebody interiors can be similarly analyzed by applying the method below to Legendrian Kirby diagrams.

\begin{proof} Let $CH_k$ be the Casson handle with all double points positive, having $k$ double points at the first stage and just one in each subsequent kinky handle. Write $X$ as the interior of a handlebody with a unique 2--handle, and let $W_{k,l}$ be its Casson-type smoothing whose Casson handle is the simplest common refinement of $CH_k$ and $CH_l$ after we reverse orientation on the latter. Then $\alpha$ is represented by an obvious immersed surface of the correct genus, with $k$ positive and $l$ negative double points, providing the required upper bounds for $\kappa(\alpha)$. To show $\kappa_+(\alpha)\ge k$, note that $W_{k,l}$ embeds in $W_{k,0}$, which for positive $k$ as given above is the Stein surface $U^+_{g(F),e(X),k}$ from the proof of Theorem~\ref{bundle}. The desired inequality follows from the adjunction inequality for immersed surfaces, Theorem~\ref{adj}. The lower bound on $\kappa_-$ follows similarly with reversed orientation. Since each $W_{k,l}$ smoothly embeds in $X$, and hence in a negative definite 4--manifold (for some orientation on  $W_{k,l}$), we obtain uncountably many diffeomorphism types as in Theorem~\ref{uncountable}. (Summing with a small exotic $\R^4$ preserves $\kappa$.)
\end{proof}

As with minimal genus, we can define kinkiness on the 2--homology of any subset $Z$ of a smooth, oriented 4--manifold, using immersions of a surface that has minimal genus for maps into $Z$ representing the given class. Its values are ordered pairs of elements of $\Z^{\ge0}\cup\{\infty\}$ (cf.\ Definition~\ref{mgsubset}). This is an invariant of myopic equivalences preserving orientation and the homology class. The method of Theorem~\ref{Fiso}, applied to the manifolds $W_{k,l}\subset X$ from the above proof, shows:

\begin{thm}\label{kinkysphere} Every orientable surface smoothly embedded in an oriented 4--manifold is topologically ambiently isotopic to almost-smooth surfaces realizing all pairs $\kappa_\pm(\alpha)\in\Z^{\ge0}\cup\{\infty\}$ satisfying the restrictions of Theorem~\ref{kink}. For a tame, topological embedding, all pairs with $\kappa_\pm(\alpha)\in\{m,\dots,\infty\}$ for a sufficiently large $m$ are similarly realized. \qed
\end{thm}

\noindent The last sentence is stronger than its analog in Theorem~\ref{Fiso} since negative double points raise the obvious upper bound for $G$ but not for $\kappa_+$. We can now complete our discussion of Corollary~\ref{infiso}:

\begin{cor}\label{infisoF} Let $F$ be a tame, orientable, compact surface in a smooth, orientable 4--manifold $X$. Then for every $n\in\Z^+$ there are $n$ almost-smooth surfaces ambiently isotopic to $F$ such that each has a neighborhood in which no neighborhood of any other embeds smoothly, preserving $[F]$ up to sign.
\end{cor}

\begin{proof} Choose surfaces $F_r$ such that $\kappa_-(F_1)<\cdots<\kappa_-(F_n)\le\kappa_+(F_n)<\cdots<\kappa_+(F_1)$. This works even for embeddings reversing orientation on $X$ or $[F]$.
\end{proof}

The first use of kinkiness \cite{kink} was rel boundary, for a circle in the boundary of a 4--manifold, such as a classical knot in $\partial B^4$ or the attaching circle of a Casson handle, and for flat topological disks (the analog of Theorem~\ref{kinkysphere}). The above techniques generalize easily to this context (and to more general rel boundary settings). We immediately obtain the most general current results on classifying Casson handles up to orientation-preserving diffeomorphism: Given a signed tree for a Casson handle $CH$, remove all negative edges, take the connected component containing the base point, and prune away all finite branches. Then the resulting valence of the base point is a lower bound on $\kappa_+(CH)$ (and on $G(CH)$), defined relative to the attaching circle, and similarly with reversed signs. For example, the Casson handles of the smoothings $W_{k.l}$ in the proof of Theorem~\ref{kink} are all distinguished (up to orientation-preserving diffeomorphism) by their kinkiness $(k,l)$. We obtain a stronger result, essentially by merging \cite{kink} with \cite[Proposition~4.1]{ch} (which produces uncountably many diffeomorphism types of Casson handles):

\begin{thm}\label{CH} For each $(k,l)\in \Z^{\ge0}\times\Z^{\ge0}-\{(0,0)\}$, there are uncountably many diffeomorphism types of Casson handles with kinkiness $(k,l)$.
\end{thm}

\begin{proof} Reversing orientation if necessary, we may assume $k>0$. For the $\R^2$--bundle $X$ over $S^2$ with $e(X)=1$, an embedding $X\to Z$ as in the proof of Theorem~\ref{large} is explicitly constructed in \cite{kink} so that the inherited smoothing of $X$ has Casson type with a unique double point (positive) in the first stage of its Casson handle. Refining, we can assume the smoothing is a refinement of $W_{k,l}$ with the same first stage as the latter. The proof of Theorem~\ref{large} (ignoring the Stein condition) gives uncountably many Casson-type diffeomorphism types on $X$ with kinkiness $(k,l)$, and hence uncountably many diffeomorphism types of Casson handles with this kinkiness. (Note that the map from oriented to unoriented diffeomorphism types is at most 2:1.)
\end{proof}

One can also define the kinkiness at infinity on classes in $\Hinf(X)$ with $g_{im}$ finite. However, this seems harder to control. While it is routine to adapt the methods of this paper to construct lower bounds, upper bounds seem more difficult since the naturally arising immersed surfaces intersect $C$. The author knows no examples of smoothings with a finite value of $\kappa_\pm$ at infinity that is strictly larger than the corresponding topological value.

\end{document}